\documentclass[11pt]{article}

\usepackage[nottoc,notlot,notlof]{tocbibind}

\usepackage{amsmath,amssymb,amsfonts,amsthm}
\usepackage{latexsym}
\usepackage{tikz}
\usetikzlibrary{automata,positioning}
\usetikzlibrary{chains,fit,shapes}
\usetikzlibrary{calc}
\usetikzlibrary{arrows}
\usepackage{tkz-graph}
\usetikzlibrary{positioning,calc}
\usetikzlibrary{graphs}
\usetikzlibrary{graphs.standard}
\usetikzlibrary{arrows,decorations.markings}

\newcount\nodecount
\tikzgraphsset{
  declare={subgraph N}%
  {
    [/utils/exec={\global\nodecount=0}]
    \foreach \nodetext in \tikzgraphV
    {  [/utils/exec={\global\advance\nodecount by1}, 
      parse/.expand once={\the\nodecount/\nodetext}] }
  },
  declare={subgraph C}%
  {
    [cycle, /utils/exec={\global\nodecount=0}]
    \foreach \nodetext in \tikzgraphV
    {  [/utils/exec={\global\advance\nodecount by1}, 
      parse/.expand once={\the\nodecount/\nodetext}] }
  }
}

\newtheorem{definition}{Definition}
\newtheorem{theorem}{Theorem}
\newtheorem{lemma}[theorem]{Lemma}
\newtheorem{remark}[theorem]{Remark}

\newtheorem{corollary}[theorem]{Corollary}

\title{Word-representability of split graphs}

\author{Sergey Kitaev\footnote{Department of Computer and Information Sciences, University of Strathclyde, 26 Richmond Street Glasgow G1 1XH, United Kingdom. 
{\bf Email:} sergey.kitaev@cis.strath.ac.uk.}, Yangjing Long\footnote{School of Mathematical Sciences,
        Shanghai Jiao Tong University, Dongchuan Road 800, 200240 Shanghai,
        China; and
        Department of Mathematics and Computer Science,
        University of Greifswald,
        Walther-Rathenau-Stra{\ss}e 47, D-17487 Greifswald, Germany. 
        {\bf Email:} yjlong@sjtu.edu.cn.}, Jun Ma\footnote{Department of Mathematics, Shanghai Jiao Tong University, Shanghai 200240, China. {\bf Email:} majun904@sjtu.edu.cn.}\ \ and Hehui Wu\footnote{Shanghai Center for Mathematical Sciences, Fudan University, 220 Handan Road, Shanghai 200433, China. {\bf Email:} hhwu@fudan.edu.cn.}}

\begin{document}  

\maketitle

\abstract{Letters $x$ and $y$ alternate in a word $w$ if after deleting in $w$ all letters but the copies of $x$ and $y$ we either obtain a word $xyxy\cdots$ (of even or odd length) or a word $yxyx\cdots$ (of even or odd length). A graph $G=(V,E)$ is word-representable if and only if there exists a word $w$
over the alphabet $V$ such that letters $x$ and $y$ alternate in $w$ if and only if $xy\in E$. It is known that a graph is word-representable if and only if it admits a certain orientation called semi-transitive orientation.

Word-representable graphs generalize several important classes of graphs such as $3$-colorable graphs, circle graphs, and comparability graphs. There is a long line of research in the literature dedicated to word-representable graphs. However, almost nothing is known on word-representability of split graphs, that is, graphs in which the vertices can be partitioned into a clique and an independent set. In this paper, we shed a light to this direction. In particular, we characterize in terms of forbidden subgraphs word-representable split graphs in which vertices in the independent set are of degree at most 2, or the size of the clique is 4. Moreover, we give necessary and sufficient conditions for an orientation of a split graph to be semi-transitive.}  

\section{Introduction}\label{sec1}

The theory of word-representable graphs is a young but very promising research area.  It was introduced by the first author in 2004 based on the joint research with Steven Seif \cite{KS08} on the celebrated {\em Perkins semigroup}, which has played a central role in semigroup theory since 1960, particularly as a source of examples and counterexamples. However, the first systematic study of word-representable graphs was not undertaken until the appearance in 2008 of the paper \cite{KP08} by the author and Artem Pyatkin, which started the development of the theory. One of the key results in the area established in \cite{HKP16} is the theorem stating that a graph is word-representable if and only if it admits a semi-transitive orientation (defined in Section~\ref{sec4}).

Up to date, nearly 20 papers have been written on the subject~\cite{K17}, and the core of the book \cite{KL15} by the first author and Vadim Lozin is devoted to the theory of word-representable graphs. It should also be mentioned that the software produced by Marc Glen \cite{G} is often of great help in dealing with word-representation of graphs.

A graph $G=(V,E)$ is {\em word-representable} if and only if there exists a word $w$
over the alphabet $V$ such that letters $x$ and $y$, $x\neq y$, alternate in $w$ if and only if $xy\in E$ (see Section~\ref{sec2} for the definition of alternating letters). The class of word-representable graphs is {\em hereditary}. That is, removing a vertex $v$ in a word-representable graph $G$ results in a word-representable graph $G'$. Indeed, if $w$ represents $G$ then $w$ with $v$ removed represents $G'$.  

We refer the Reader to \cite{KL15}, where relevance of word-representable graphs to various fields is explained, thus providing a motivation to study the graphs. These fields are algebra, graph theory, computer science, combinatorics on words, and scheduling. In particular, word-representable graphs are important from graph-theoretical point of view, since they generalize several fundamental classes of graphs (e.g.\ {\em circle graphs}, {\em $3$-colorable graphs} and {\em comparability graphs}).

Even though much is understood about word-representable graphs \cite{K17,KL15}, almost nothing is known on word-representability of {\em split graphs}, that is, graphs in which the vertices can be partitioned into a clique and an independent set. The only known examples in the literature of non-word-representable split graphs are shown in Figure~\ref{nonRepTri}. These graphs are three out of the four graphs on the last line in Figure 3.9 on page 48 in \cite{KL15} showing all 25 non-word-representable graphs on 7 vertices. We note that non-representability of $T_1$ is discussed, e.g. in \cite{CKS16}, and non-word-representability of $T_2$ follows from Theorem~\ref{neighbourhood} below. The minimality by the number of vertices for the graphs follows from the fact that only the {\em wheel graph} $W_5$ (see Section~\ref{sec2} for the definition) is non-word-representable on six vertices.

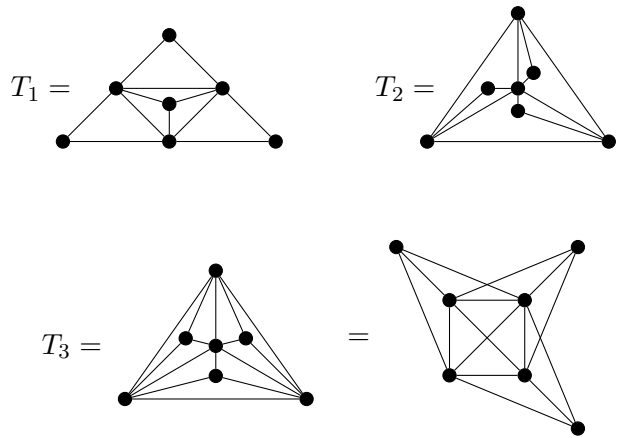
\begin{figure}
\begin{center}

\begin{tabular}{ccccc}
\begin{tikzpicture}[node distance=1cm,auto,main node/.style={fill,circle,draw,inner sep=0pt,minimum size=5pt}]

\node[main node] (1) {};
\node[main node] (2) [below left of=1] {};
\node[main node] (3) [below right of=1] {};
\node[main node] (4) [below left of=2] {};
\node[main node] (5) [below right of=2] {};
\node[main node] (6) [below right of=3] {};
\node[main node] (7) [above of=5, yshift=-0.5cm] {};

\node (8) [left of=2] {$T_1=$};

\path
(1) edge (4)
(1) edge (6);

\path
(5) edge (2)
(5) edge (3)
(5) edge (4)
(5) edge (6);

\path
(7) edge (2)
(7) edge (3)
(7) edge (5);

\path
(2) edge (3);
\end{tikzpicture}

& 

\ \ \ 

&

\begin{tikzpicture}[node distance=1cm,auto,main node/.style={fill,circle,draw,inner sep=0pt,minimum size=5pt}]

\node[main node] (1) {};
\node[main node] (2) [below of=1] {};
\node[main node] (3) [below left of=2, xshift=-0.5cm] {};
\node[main node] (4) [below right of=2, xshift=0.5cm] {};
\node[main node] (5) [left of=2, xshift=0.6cm] {};
\node[main node] (6) [below of=2, yshift=0.7cm] {};
\node[main node] (7) [above right of=2, xshift=-0.5cm, yshift=-0.5cm] {};
\node (8) [left of=2,  xshift=-0.5cm] {$T_2=$};

\path
(1) edge (2)
(1) edge (4)
(1) edge (7)
(1) edge (3);

\path
(2) edge (3)
(2) edge (4)
(2) edge (5)
(2) edge (6)
(2) edge (7);

\path
(3) edge (4)
(3) edge (5);

\path
(6) edge (4);

\end{tikzpicture}

\end{tabular}

\ \\[1cm]

\begin{tikzpicture}[node distance=1cm,auto,main node/.style={fill,circle,draw,inner sep=0pt,minimum size=5pt}]

\node[main node] (1) {};
\node[main node] (2) [below of=1] {};
\node[main node] (3) [below left of=2,xshift=-0.5cm] {};
\node[main node] (4) [below right of=2,xshift=0.5cm] {};
\node[main node] (5) [below of=2,yshift=0.6cm] {};
\node[main node] (6) [left of=2,yshift=0.1cm,xshift=0.6cm] {};
\node[main node] (7) [right of=2,yshift=0.1cm,xshift=-0.6cm] {};

\node [above left of=3] {$T_3=$};

\path
(2) edge (6)
(2) edge (7)
(5) edge (3)
(5) edge (4)
(5) edge (2)
(1) edge (3)
(1) edge (4)
(6) edge (3)
(7) edge (4)
(1) edge (6)
(1) edge (7)
(2) edge (3)
(2) edge (4)
(4) edge (3)
(1) edge (2);

\node [right of=7,xshift=0.5cm] {$=$};

\node[main node] (8) [above right of=7,xshift=2cm,yshift=-0.2cm] {};
\node[main node] (9) [right of=8] {};
\node[main node] (10) [below of=9] {};
\node[main node] (11) [left of=10] {};
\node[main node] (12) [above left of=8] {};
\node[main node] (13) [above right of=9] {};
\node[main node] (14) [below right of=10] {};

\path
(8) edge (9)
(8) edge (10)
(8) edge (11)
(9) edge (10)
(9) edge (11)
(10) edge (11)
(12) edge (8)
(12) edge (9)
(12) edge (11)
(13) edge (8)
(13) edge (9)
(13) edge (10)
(14) edge (11)
(14) edge (9)
(14) edge (10);

\end{tikzpicture}

\caption{\label{nonRepTri} The minimal (by the number of vertices) non-word-representable split graphs $T_1$, $T_2$ and $T_3$}
\end{center}
\end{figure}

In this paper we characterize in terms of forbidden subgraphs word-representable split graphs  in which vertices in the independent set are of degree at most 2 (see Theorem~\ref{main-1}), or the size of the clique is 4 (see Theorem~\ref{main-2}). To achieve these results, we introduce the following classes of graphs:
\begin{itemize}
\item $K^{\triangle}_\ell$, $\ell\geq 3$, in Definition~\ref{def-K-triang} that are always word-representable by Theorem~\ref{thm-K-Tri-m-w-r}. This class of graphs is generalized in Corollary~\ref{K-ell-k} to word-representable graphs $K_{\ell}^k$. 
\item $A_\ell$, $\ell\geq 4$, in Definition~\ref{def-Aell} that are always minimal non-word-representable by Theorem~\ref{Aell-min-non-repres}. This class of graphs generalizes the known non-word-representable graph $T_1$ in Figure~\ref{nonRepTri}, which corresponds to $\ell=4$. \end{itemize} 
Also, in Theorem~\ref{main-orientation} we give necessary and sufficient conditions for an orientation of a split graph to be semi-transitive. A particular property of semi-transitive orientations is established in Theorem~\ref{intercahnging}.

Finally, directions for further research can be found in Section~\ref{last-sec}.

\section{Split graphs}

Let $S_n$ be a split graph on $n$ vertices. The vertices of $S_n$ can be 
partitioned into a maximal clique $K_m$ and an independent set $E_{n-m}$, that is,  
the vertices in $E_{n-m}$ are of degree at most $m-1$. We only consider such ``maximal'' partitions throughout the paper and say that $S_n=(E_{n-m},K_m)$. 

Based on \cite{FH77} it can be shown \cite[Theorem 2.2.10]{KL15} that the class of split graphs is the intersection of the classes of {\em chordal graphs} (those avoiding all cycle graphs $C_m$, $m\geq 4$, as induced subgraphs) and their complements, and this is precisely the class of graphs not containing the graphs $C_4$, $C_5$ and $2K_2=\begin{tikzpicture}[node distance=0.3cm,auto,main node/.style={fill,circle,draw,inner sep=0pt,minimum size=4pt}]
\node[main node] (1) {};
\node[main node] (2) [below of=1] {};
\node[main node] (3) [right of=1] {};
\node[main node] (4) [right of=2] {};
\path
(1) edge (2);
\path
(3) edge (4);
\end{tikzpicture}
$ as induced subgraphs. More relevant to our studies is the following result (see Section~\ref{sec2} for the definition of a comparability graph).

\begin{theorem}[\cite{Gol}]\label{thm-comp-split} Split comparability graphs are characterized by avoiding the three 
graphs in Figure~\ref{forbidden-split-comp} as induced subgraphs. \end{theorem}

\begin{figure}[h]
\begin{center}
\begin{tabular}{ccccc}
\begin{tikzpicture}[node distance=0.8cm,auto,main node/.style={fill,circle,draw,inner sep=0pt,minimum size=5pt}]

\node[main node] (1) {};
\node[main node] (2) [below of=1] {};
\node[main node] (3) [below left of=2] {};
\node[main node] (4) [below right of=2] {};
\node[main node] (5) [below of=3] {};
\node[main node] (6) [below of=4] {};
\node (8) [left of=2] {$B_1=$};
\path
(3) edge (2)
(3) edge (4)
(3) edge (5);

\path
(4) edge (2)
(4) edge (6);

\path
(1) edge (2);
\end{tikzpicture}

& 

&

\begin{tikzpicture}[node distance=1cm,auto,main node/.style={fill,circle,draw,inner sep=0pt,minimum size=5pt}]

\node[main node] (1) {};
\node[main node] (2) [below left of=1] {};
\node[main node] (3) [below right of=1] {};
\node[main node] (4) [below left of=2] {};
\node[main node] (5) [below right of=2] {};
\node[main node] (6) [below right of=3] {};
\node (8) [left of=2] {$B_2=$};

\path
(1) edge (4)
(1) edge (6);

\path
(5) edge (2)
(5) edge (3)
(5) edge (4)
(5) edge (6);

\path
(2) edge (3);
\end{tikzpicture}

& 

&

\begin{tikzpicture}[node distance=1cm,auto,main node/.style={fill,circle,draw,inner sep=0pt,minimum size=5pt}]

\node[main node] (1) {};
\node[main node] (2) [right of=1] {};
\node[main node] (3) [right of=2] {};
\node[main node] (4) [below left of=2] {};
\node[main node] (5) [below right of=2] {};
\node[main node] (6) [below of=4] {};
\node[main node] (7) [below of=5] {};
\node (8) [left of=4] {$B_3=$};

\path
(1) edge (3)
(1) edge (4);

\path
(2) edge (4)
(2) edge (5);

\path
(4) edge (5)
(4) edge (6);

\path
(5) edge (3)
(5) edge (7);
\end{tikzpicture}

\end{tabular}
\end{center}
\vspace{-5mm}
\caption{Forbidden induced subgraphs for split comparability graphs}\label{forbidden-split-comp}
\end{figure}
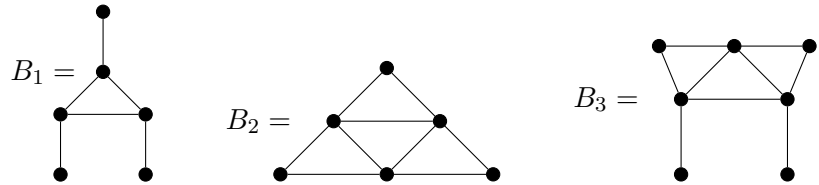

It is known that any comparability graph is word-representable, and such a graph on $n$ vertices can be represented by a word, which is a concatenation of (several) permutations of length $n$ \cite{KS08,KL15}. Thus, when studying word-representability of a split graph, we can assume that one of the graphs in Figure~\ref{forbidden-split-comp} is present as an induced subgraph, because otherwise the split graph in question is a comparability graph and thus is word-representable. 

\section{Word-Representable graphs}\label{sec2}

Suppose that $w$ is a word over some alphabet and $x$ and $y$ are two distinct letters in $w$. We say that $x$ and $y$ {\em alternate} in $w$ if after deleting in $w$ {\em all} letters {\em but} the copies of $x$ and $y$ we either obtain a word $xyxy\cdots$ (of even or odd length) or a word $yxyx\cdots$ (of even or odd length). For example, in the word 23125413241362, the letters 2 and 3 alternate. So do the letters 5 and 6, while the letters 1 and 3 do {\em not} alternate.  

\begin{definition}\label{wrg-def} A graph $G=(V,E)$ is {\em word-representable} if and only if there exists a word $w$
over the alphabet $V$ such that letters $x$ and $y$, $x\neq y$, alternate in $w$ if and only if $xy\in E$. (By definition, $w$ {\em must} contain {\em each} letter in $V$.) We say that $w$ {\em represents} $G$, and that $w$ is a {\em word-representant}. \end{definition}

Definition~\ref{wrg-def} works for both vertex-labeled and unlabeled graphs because any labeling of a graph $G$ is equivalent to any other labeling of $G$ with respect to word-representability (indeed, the letters of a word $w$ representing $G$ can always be renamed). For example, the graph to the left in Figure~\ref{wrg-ex} is word-representable because its labeled version to the right in Figure~\ref{wrg-ex} can be represented by 1213423. For another example, each {\em complete graph} $K_n$ can be represented by any permutation $\pi$ of $\{1,2,\ldots,n\}$, or by $\pi$ concatenated any number of times.   Also, the {\em empty graph} $E_n$ (also known as {\em edgeless graph}, or {\em null graph}) on vertices $\{1,2,\ldots,n\}$ can be represented by $12\cdots (n-1)nn(n-1)\cdots 21$, or by any other permutation concatenated with the same permutation written in the reverse order.

\begin{figure}[h]
\begin{center}
\begin{tabular}{ccc}
\begin{tikzpicture}[node distance=1cm,auto,main node/.style={fill,circle,draw,inner sep=0pt,minimum size=5pt}]

\node[main node] (1) {};
\node[main node] (2) [below left of=1] {};
\node[main node] (3) [below right of=1] {};
\node[main node] (4) [below right of=2] {};

\path
(1) edge (2)
(1) edge (3);

\path
(2) edge (3)
(2) edge (4);
\end{tikzpicture}

& 

\ \ \ \ \ \ \

&

\begin{tikzpicture}[node distance=1cm,auto,main node/.style={circle,draw,inner sep=1pt,minimum size=2pt}]

\node[main node] (1) {{\tiny 3}};
\node[main node] (2) [below left of=1] {{\tiny 2}};
\node[main node] (3) [below right of=1] {{\tiny 4}};
\node[main node] (4) [below right of=2] {{\tiny 1}};

\path
(1) edge (2)
(1) edge (3);

\path
(2) edge (3)
(2) edge (4);

\end{tikzpicture}

\end{tabular}
\end{center}
\vspace{-5mm}
\caption{An example of a word-representable graph}\label{wrg-ex}
\end{figure}
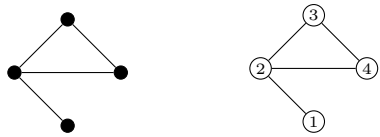

An orientation of a graph is {\em transitive} if presence of edges $u\rightarrow v$ and $v\rightarrow z$ implies presence of the edge $u\rightarrow z$.  An unoriented graph is a {\em comparability graph} if it admits a transitive orientation. It is well known \cite[Section 3.5.1]{KL15}, and is not difficult to show that the smallest non-comparability graph is the cycle graph $C_5$.

The following important result reveals the structure of neighbourhoods of vertices in a word-representable graph.

\begin{theorem}[\cite{KP08}]\label{neighbourhood} If a graph $G$ is word-representable then the neighbourhood of each vertex in $G$ is a comparability graph. \end{theorem}

Theorem~\ref{neighbourhood} allows to construct examples of non-word-representable graphs. For example, the {\em wheel graph} $W_5$, obtained from the cycle graph $C_5$ by adding an apex (all-adjacent vertex) is the minimum (by the number of vertices) non-word-representable graph. As mentioned above, $W_5$ is the only non-word-representable graph on 6 vertices.  

The converse to Theorem~\ref{neighbourhood} is {\em not} true as demonstrated by the counterexamples in Figure~\ref{2-counterex} taken from \cite{HKP10} and \cite{CKL17}, respectively.

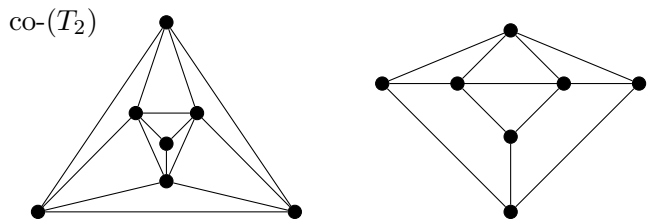
\begin{figure}
\begin{center}
\begin{tabular}{ccc}

\begin{tikzpicture}[node distance=1cm,auto,main node/.style={fill,circle,draw,inner sep=0pt,minimum size=5pt}]

\node[main node] (1) {};
\node[main node] (2) [above left of=1,yshift=-0.3cm,xshift=0.3cm] {};
\node[main node] (3) [above right of=1,yshift=-0.3cm,xshift=-0.3cm] {};
\node[main node] (4) [below of=1,yshift=0.5cm] {};
\node[main node] (5) [above right of=2,xshift=-0.3cm,yshift=0.5cm] {};
\node[main node] (6) [below right of=4,xshift=1cm,yshift=0.3cm] {};
\node[main node] (7) [below left of=4,xshift=-1cm,yshift=0.3cm] {};
\node (8) [left of=5,xshift=-0.5cm] {co-($T_2$)};

\path
(5) edge (7) 
(5) edge (2) 
(5) edge (3) 
(5) edge (6) 
(2) edge (3) 
(2) edge (1) 
(2) edge (4) 
(2) edge (7) 
(1) edge (3) 
(1) edge (4) 
(6) edge (3) 
(6) edge (4) 
(6) edge (7)
(4) edge (7)
(4) edge (3);

\end{tikzpicture}

&
\ \ 
&

\begin{tikzpicture}[node distance=1cm,auto,main node/.style={fill,circle,draw,inner sep=0pt,minimum size=5pt}]
\node[main node] (1) {};
\node[main node] (2) [below left of=1] {};
\node[main node] (3) [below right of=1] {};
\node[main node] (4) [left of=2] {};
\node[main node] (5) [right of=3] {};
\node[main node] (6) [below right of=2] {};
\node[main node] (7) [below of=6] {};

\path
(4) edge (5) 
(1) edge (5) 
(1) edge (4) 
(1) edge (2) 
(1) edge (3) 
(6) edge (2) 
(6) edge (3) 
(6) edge (7) 
(7) edge (4) 
(7) edge (5);

\end{tikzpicture}

\end{tabular}
\end{center}
\vspace{-5mm}
\caption{Non-word-representable graphs in which each neighbourhood is a comparability graph}\label{2-counterex}
\end{figure}

\section{Semi-transitive orientations}\label{sec4}

A {\em shortcut} is an {\em acyclic non-transitively oriented} graph obtained from a directed cycle graph forming a directed cycle on at least four vertices by changing the orientation of one of the edges, and possibly by adding more directed edges connecting some of the vertices (while keeping the graph be acyclic and non-transitive). Thus, any shortcut  

\begin{itemize}

\item is {\em acyclic} (that it, there are {\em no directed cycles});

\item has {\em at least} 4 vertices;

\item has {\em exactly one} source (the vertex with no edges coming in), {\em exactly one} sink (the vertex with no edges coming out), and a {\em directed path} from the source to the sink that goes through {\em every} vertex in the graph;

\item has an edge connecting the source to the sink that we refer to as the {\em shortcutting edge};

\item is {\em not} transitive (that it, there exist vertices $u$, $v$ and $z$ such that $u\rightarrow v$ and $v\rightarrow z$ are edges, but there is {\em no} edge $u\rightarrow z$).

\end{itemize}

\begin{definition} An orientation of a graph is {\em semi-transitive} if it is {\em acyclic} and 
{\em shortcut-free}.\end{definition}

It is easy to see from definitions that {\em any} transitive orientation is necessary  semi-transitive. The converse is {\em not} true. Thus semi-transitive orientations generalize transitive orientations. We will use the following simple lemma. 

\begin{lemma}\label{lem-tran-orie} Let $K_m$ be a clique in a graph $G$. Then any acyclic orientation of $G$ induces a transitive orientation on $K_m$. In particular, any semi-transitive orientation of $G$ induces a transitive orientation on $K_m$. In either case, the orientation induced on $K_m$ contains a single source and a single sink. \end{lemma}

\begin{proof} Oriented $K_m$ is called a tournament, and it is well known, and is not difficult to prove that any tournament contains a Hamiltonian path, that is, a path going through each vertex exactly once. Taking into account that the orientation of $K_m$ is acyclic, it must be transitive with the unique source and sink given by the Hamiltonian path. \end{proof}

A key result in the theory of word-representable graphs is the following theorem. 

\begin{theorem}[\cite{HKP16}]\label{key-thm} A graph is word-representable if and only if it admits a semi-transitive orientation. \end{theorem}

A corollary of Theorem~\ref{key-thm} is the following useful for us theorem.

\begin{theorem}[\cite{HKP16}]\label{3-col-thm} Any $3$-colorable graph is word-representable. \end{theorem}

\begin{remark} By Theorem~\ref{3-col-thm} below we can assume that $m\geq 4$, because otherwise $S_n$ is $3$-colorable and thus is word-representable. \end{remark}

\section{Preliminaries}

We begin with a result that allows us to assume in our studies that the size of a maximal clique in a split graph is at least 4.

\begin{theorem} Let $S_n=(E_{n-m},K_m)$ be a split graph and $m\leq 3$. Then $S_n$ is word-representable. \end{theorem}

\begin{proof} $S_n$ is clearly 3-colorable, and thus, by Theorem~\ref{3-col-thm}, it is word-representable.  \end{proof}

The following lemma allows us to assume in our studies that (i) each vertex in a split graph is of degree at least 2, and (ii) no two vertices have the same set of neighbours. 

\begin{lemma}\label{lemma-assumptions} Let $S_n=(E_{n-m},K_m)$ be a split graph, and a spit graph $S_{n+1}$ is obtained from $S_n$ by either adding a vertex of degree $0$ (to $E_{n-m}$), or adding a vertex of degree $1$ (to $E_{n-m}$), or by ``coping'' a vertex (either in $E_{n-m}$ or in $K_m$), that is, by adding a vertex whose neighbourhood is identical to the neighbourhood of a vertex in $S_n$. Then $S_n$ is word-representable if and only if $S_{n+1}$ is word-representable. \end{lemma}

\begin{proof} Suppose a vertex $x$ of degree $0$ is added to a word-representable graph $S_n$ having a word-representant $w$. Then the word $xxw$ represents $S_{n+1}$.  

Connecting two word-representable graphs by an edge gives a word-representable graph (see \cite[Section 5.4.3]{KL15}), which is easy to see using semi-transitive orientations and Theorem~\ref{key-thm}. The one vertex graph is word-representable, so the lemma is true for adding a vertex of degree~1. 

Copying a vertex $v$ in $S_n$ (either connected, or not, to $v$) is a particular case of replacing any vertex in a word-representable graph by a {\em module}, which is a comparability graph. It is known (see \cite[Section 5.4.4]{KL15}) that such a replacement gives a word-representable graph, which completes the proof of the lemma. \end{proof}

\begin{definition}\label{def-K-triang} For $\ell\geq 3$, the graph $K^{\triangle}_\ell$ is obtained from the complete graph $K_\ell$ labeled by $1,2,\ldots, \ell$, by adding a vertex $i'$ of degree $2$ connected to vertices $i$ and $i+1$ for each $i\in \{1,2,\ldots,\ell-1\}$. Also, a vertex $\ell'$ connected to the vertices $1$ and $\ell$ is added. See Figure~\ref{graph-K-triang}  for the graph $K^{\triangle}_6$. \end{definition}

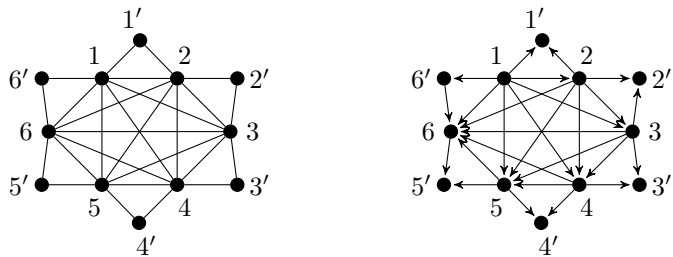
\begin{figure}[h]
\begin{center}
\begin{tabular}{ccc}
\begin{tikzpicture}[node distance=1cm,auto,main node/.style={fill,circle,draw,inner sep=0pt,minimum size=5pt}]

\node[main node] (1) {};
\node[main node] (2) [right of=1] {};
\node[main node] (3) [below right of=2] {};
\node[main node] (4) [below left of=3] {};
\node[main node] (5) [left of=4] {};
\node[main node] (6) [above left of=5] {};

\node[main node] (7) [above right of=1, xshift=-0.2cm, yshift=-0.2cm] {};
\node[main node] (8) [right of=2, xshift=-0.2cm] {};
\node[main node] (9) [right of=4, xshift=-0.2cm] {};
\node[main node] (10) [below left of=4, xshift=0.2cm, yshift=0.2cm] {};
\node[main node] (11) [left of=5, xshift=0.2cm] {};
\node[main node] (12) [left of=1, xshift=0.2cm] {};

\node [above of=1, yshift=-0.7cm,xshift=-0.1cm] {{\small 1}};
\node [above of=2, yshift=-0.7cm,xshift=0.1cm] {{\small 2}};
\node [right of=3, xshift=-0.7cm] {{\small 3}};
\node [below of=4, yshift=0.7cm,xshift=0.1cm] {{\small 4}};
\node [below of=5, yshift=0.7cm,xshift=-0.1cm] {{\small 5}};
\node [left of=6, xshift=0.7cm] {{\small 6}};

\node [above of=7, yshift=-0.7cm,xshift=-0.1cm] {{\small $1'$}};
\node [right of=8, xshift=-0.7cm] {{\small $2'$}};
\node [right of=9, xshift=-0.7cm] {{\small $3'$}};
\node [below of=10, yshift=0.7cm,xshift=0.1cm] {{\small $4'$}};
\node [left of=11, xshift=0.7cm] {{\small $5'$}};
\node [left of=12, xshift=0.7cm] {{\small $6'$}};

\path
(1) edge (2)
(1) edge (3)
(1) edge (4)
(1) edge (5)
(1) edge (6)
(1) edge (12)
(1) edge (7);

\path
(2) edge (3)
(2) edge (4)
(2) edge (5)
(2) edge (6)
(2) edge (7)
(2) edge (8);

\path
(3) edge (4)
(3) edge (5)
(3) edge (6)
(3) edge (8)
(3) edge (9);

\path
(4) edge (5)
(4) edge (6)
(4) edge (9)
(4) edge (10);

\path
(5) edge (6)
(5) edge (10)
(5) edge (11);

\path
(6) edge (11)
(6) edge (12);

\end{tikzpicture}

& 

\ \ \ \ \ \ \

&

\begin{tikzpicture}[->,>=stealth', shorten >=1pt, node distance=1cm,auto,main node/.style={fill,circle,draw,inner sep=0pt,minimum size=5pt}]

\node[main node] (1) {};
\node[main node] (2) [right of=1] {};
\node[main node] (3) [below right of=2] {};
\node[main node] (4) [below left of=3] {};
\node[main node] (5) [left of=4] {};
\node[main node] (6) [above left of=5] {};

\node[main node] (7) [above right of=1, xshift=-0.2cm, yshift=-0.2cm] {};
\node[main node] (8) [right of=2, xshift=-0.2cm] {};
\node[main node] (9) [right of=4, xshift=-0.2cm] {};
\node[main node] (10) [below left of=4, xshift=0.2cm, yshift=0.2cm] {};
\node[main node] (11) [left of=5, xshift=0.2cm] {};
\node[main node] (12) [left of=1, xshift=0.2cm] {};

\node [above of=1, yshift=-0.7cm,xshift=-0.1cm] {{\small 1}};
\node [above of=2, yshift=-0.7cm,xshift=0.1cm] {{\small 2}};
\node [right of=3, xshift=-0.7cm] {{\small 3}};
\node [below of=4, yshift=0.7cm,xshift=0.1cm] {{\small 4}};
\node [below of=5, yshift=0.7cm,xshift=-0.1cm] {{\small 5}};
\node [left of=6, xshift=0.7cm] {{\small 6}};

\node [above of=7, yshift=-0.7cm,xshift=-0.1cm] {{\small $1'$}};
\node [right of=8, xshift=-0.7cm] {{\small $2'$}};
\node [right of=9, xshift=-0.7cm] {{\small $3'$}};
\node [below of=10, yshift=0.7cm,xshift=0.1cm] {{\small $4'$}};
\node [left of=11, xshift=0.7cm] {{\small $5'$}};
\node [left of=12, xshift=0.7cm] {{\small $6'$}};

\path
(1) edge (2)
(1) edge (3)
(1) edge (4)
(1) edge (5)
(1) edge (6)
(1) edge (12)
(1) edge (7);

\path
(2) edge (3)
(2) edge (4)
(2) edge (5)
(2) edge (6)
(2) edge (7)
(2) edge (8);

\path
(3) edge (4)
(3) edge (5)
(3) edge (6)
(3) edge (8)
(3) edge (9);

\path
(4) edge (5)
(4) edge (6)
(4) edge (9)
(4) edge (10);

\path
(5) edge (6)
(5) edge (10)
(5) edge (11);

\path
(6) edge (11);

\path
(12) edge (6);

\end{tikzpicture}

\end{tabular}
\end{center}
\vspace{-5mm}
\caption{The graph $K^{\triangle}_6$ and one of its semi-transitive orientations}\label{graph-K-triang}
\end{figure}

\begin{theorem}\label{thm-K-Tri-m-w-r} $K^{\triangle}_\ell$ is word-representable. \end{theorem}

\begin{proof} In the case of odd $\ell$, it is not difficult to come up with a word representing $K^{\triangle}_\ell$ based on the representation $12\cdots \ell12\cdots \ell$ of $K_\ell$ and  adding $i'$s as follows (where we present the resulting word on two lines):
$$\begin{array}{l}1'121'3'343'5'565'\cdots (\ell-2)'(\ell-2)(\ell-1)(\ell-2)'\ell'\ell1\ell'2'232' \\ 4'454' \cdots (\ell-1)'(\ell-1)\ell(\ell-1)'.\end{array}$$
However, we next provide a semi-transitive orientation of $K^{\triangle}_\ell$ that works for any $\ell$, so that the statement will follow from Theorem~\ref{key-thm}.

First, orient the $K_\ell$ transitively so that there is a directed path $1\rightarrow 2\rightarrow \cdots\rightarrow \ell$ as shown for the case $\ell=6$ in Figure~\ref{graph-K-triang}. Next, for $i\in \{1,2,\ldots,\ell-1\}$ orient the edges incident to $i'$ as $i\rightarrow i'$ and $(i+1)\rightarrow i'$. Finally, orient the edges incident to $m'$ as $1\rightarrow \ell'$ and $\ell'\rightarrow \ell$ as again shown for the case $\ell=6$ in Figure~\ref{graph-K-triang} (see Remark~\ref{orient-KmTriang} about other ways to orient $K^{\triangle}_\ell$ semi-transitively). 

We claim that the orientation obtained is semi-transitive. Indeed, it is easy to see that there are no directed cycles. Furthermore, because $K_\ell$ is transitively oriented, any possible shortcut must involve a vertex $i'$. Clearly, $\ell'\rightarrow \ell$ and $1\rightarrow \ell'$ are  not shortcutting edges because $\ell'$ is neither a sink nor a source. Note that $a<b$ whenever $a\rightarrow b$ for $a,b\in\{1,2,\ldots,\ell\}$. Using this observation, $(i+1)\rightarrow i'$, for $i\in\{1,2,\ldots,\ell-1\}$, is not a shortcutting edge because there is no path from a vertex $(i+1)$ to a vertex $i$. Finally, $i\rightarrow i'$, for $i\in\{1,2,\ldots,\ell-1\}$, cannot be a shortcutting edge because there is no path of length larger than 2 from a vertex $i$ to a vertex $i'$. \end{proof}

\begin{remark}\label{orient-KmTriang} We note that if the orientation of the $K_\ell$ in a $K^{\triangle}_\ell$ is fixed as in the proof of Theorem~\ref{thm-K-Tri-m-w-r}, there are exactly $2^{\ell-1}$ ways to extend the orientation of $K_\ell$ to that of $K^{\triangle}_\ell$ in a semi-transitive way. Indeed, in the proof of Theorem~\ref{thm-K-Tri-m-w-r}, the vertices $i'$ for $i\in\{1,2,\ldots,m-1\}$ were made sinks, but any of these could be made sources (a similar argument as that in the proof of Theorem~\ref{thm-K-Tri-m-w-r} would show that the obtained orientation would be semi-transitive). It can then be shown that any of the three remaining ways to orient the edges $1\ell'$ and $\ell\ell'$ will either result in a directed cycle, or a shortcut. The same situation is with any of the two remaining ways to orient the edges $ii'$ and $(i+1)i$ for $i\in\{1,2,\ldots,\ell-1\}$. For example, making $\ell'$ to be a sink, the edge $1\rightarrow \ell'$ will become a shortcutting edge (the vertices $1$, $\ell-1$, $\ell$ and $\ell'$ induce a shortcut in this case).\end{remark}

\begin{definition}\label{def-Aell} For $\ell\geq 4$, let $A_\ell$ be the graph obtained from $K^{\triangle}_{\ell-1}$ by adding a vertex $\ell$ connected to the vertices $1,2,\ldots,\ell-1$ and no other vertices.  Note that $A_4=T_1$ in Figure~\ref{nonRepTri}. A schematic way to represent a graph $A_\ell$ is shown in Figure~\ref{graph-Aell}. \end{definition}

\begin{theorem}\label{Aell-min-non-repres} $A_\ell$ is a minimal non-word-representable graph. \end{theorem}

\begin{figure}[h]
\begin{center}
\begin{tikzpicture}[node distance=1cm,auto,main node/.style={fill,circle,draw,inner sep=0pt,minimum size=5pt}]

\node[main node] (1) {};
\node[main node] (2) [right of=1] {};
\node[main node] (3) [below right of=2] {};
\node[main node] (4) [below left of=3] {};
\node[main node] (5) [left of=4] {};
\node[main node] (6) [above left of=5] {};
\node[main node] (13) [below right of=1, xshift=-0.2cm] {};

\node[main node] (7) [above right of=1, xshift=-0.2cm, yshift=-0.2cm] {};
\node[main node] (8) [right of=2, xshift=-0.2cm] {};
\node[main node] (9) [right of=4, xshift=-0.2cm] {};
\node[main node] (10) [below left of=4, xshift=0.2cm, yshift=0.2cm] {};
\node[main node] (11) [left of=5, xshift=0.2cm] {};
\node[main node] (12) [left of=1, xshift=0.2cm] {};

\node [above of=1, yshift=-0.7cm,xshift=-0.1cm] {{\small 1}};
\node [above of=2, yshift=-0.7cm,xshift=0.1cm] {{\small 2}};
\node [below of=5, yshift=0.7cm,xshift=-0.1cm] {{\small $y$}};
\node [left of=6, xshift=0.4cm] {{\small $\ell-1$}};

\node [right of=9, xshift=-0.7cm] {{\small $x$}};

\node [above of=13, yshift=-0.7cm,xshift=0.1cm] {{\small $\ell$}};

\path
(1) edge (2)
(1) edge (6)
(1) edge (12)
(1) edge (7);

\path
(2) edge (3)
(2) edge (7)
(2) edge (8);

\path
(3) edge (4)
(3) edge (8)
(3) edge (9);

\path
(4) edge (5)
(4) edge (9)
(4) edge (10);

\path
(5) edge (6)
(5) edge (10)
(5) edge (11);

\path
(6) edge (11)
(6) edge (12);

\end{tikzpicture}

\end{center}
\vspace{-5mm}
\caption{A schematic way to represent $A_\ell$}\label{graph-Aell}
\end{figure}
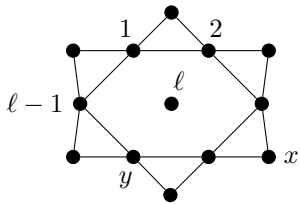

\begin{proof} \noindent
{\bf Minimality.} Because of the symmetries, we only need to consider three cases with a reference to Figure~\ref{graph-Aell}. 

\begin{itemize}
\item Removing the vertex $\ell$ we obtain the graph $K^{\triangle}_{\ell-1}$ which is word-representable by Theorem~\ref{thm-K-Tri-m-w-r}.
\item Removing the vertex $x$ we get a graph isomorphic to the graph obtained from $K^{\triangle}_\ell$ by removing the vertices $1'$ and $2'$. Such a graph is word-representable by  Theorem~\ref{thm-K-Tri-m-w-r} taking into account the hereditary nature of word-representability.
\item Removing the vertex $y$ we get a graph isomorphic to the graph obtained from $K^{\triangle}_{\ell-1}$ by removing the vertex $1'$, which is word-representable by  Theorem~\ref{thm-K-Tri-m-w-r} taking into account the hereditary nature of word-representability.
\end{itemize}

\noindent
{\bf Non-word-representability.} We will show that $A_\ell$ does not admit a semi-transitive orientation, and the result will follow by Theorem~\ref{key-thm}. 

Suppose $A_\ell$ admits a semi-transitive orientation. By Lemma~\ref{lem-tran-orie}, this orientation induces a transitive orientation on the clique of size $\ell-1$ obtained by removing the vertex $\ell$. We claim that without loss of generality, we can assume that the Hamiltonian path on this clique is $1\rightarrow 2\rightarrow \cdots \rightarrow (\ell-1)$, or its cyclic shift (e.g. $2\rightarrow 3 \rightarrow \cdots \rightarrow (\ell-1) \rightarrow 1$, or $3\rightarrow 4 \rightarrow \cdots \rightarrow (\ell-1) \rightarrow 1 \rightarrow 2$, etc). Indeed, if that would not be the case, then changing all orientations to the opposite, if necessary, there must exist $i$ such that 
\begin{itemize}
\item $P_i=i\rightarrow x_1\rightarrow x_2 \rightarrow \cdots \rightarrow x_j \rightarrow (i+1)$ is part of the Hamiltonian path for $j\geq 1$; if $i=(\ell-1)$ then $(i+1):=1$;
\item either $x\rightarrow i$, or $(i+1)\rightarrow y$, or both, are present in the Hamiltonian path for some vertices $x$ and $y$.
\end{itemize}

If the orientation of the edge $i'(i+1)$ is $(i+1) \rightarrow i'$ then this edge, along with $P_i$ and the edge $ii'$ will either induce a directed cycle, or a shortcut; contradiction. Thus, the orientation of $i'(i+1)$ must be $i' \rightarrow (i+1)$. Furthermore, to avoid a shortcut involving the edge $i' \rightarrow (i+1)$ and $P_i$, we must orient the edge $ii'$ as $i\rightarrow i'$. But now, the graph induced by $P_i$, $i\rightarrow i'$, $i' \rightarrow (i+1)$, and $x\rightarrow i$ or $(i+1)\rightarrow y$ (whatever exists) will induce a shortcut. Indeed, in the former case, the edge $x\rightarrow (i+1)$ is present, but the edge $x\rightarrow i'$ is not, while in the latter case, the edge $i\rightarrow y$ is present, while $i'\rightarrow y$ is not. Thus, renaming the vertices, if necessary (which is equivalent to a cyclic shift), we can assume that the partial orientation of the semi-transitively oriented $A_\ell$ is as in the graph to the left in  Figure~\ref{fig-nwr-Aell}, where we do not draw the edges $i\rightarrow j$ for $|j-i|\geq 2$, except for the edge $1\rightarrow (\ell-2)$, to arrange a better look for the figure (although existence of these edges is assumed).

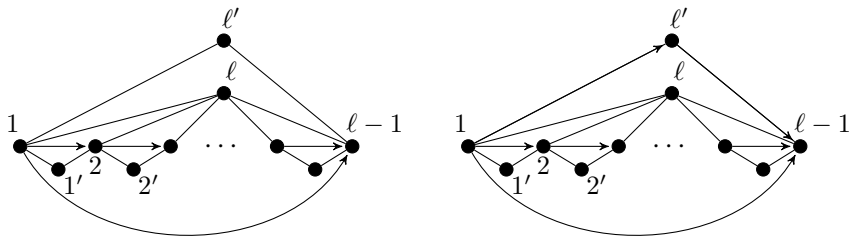
\begin{figure}[h]
\begin{center}
\begin{tabular}{cc}
\begin{tikzpicture}[node distance=1cm,auto,main node/.style={fill,circle,draw,inner sep=0pt,minimum size=5pt}]

\node[main node] (1) {};
\node[main node] (2) [right of=1] {};
\node[main node] (3) [right of=2] {};

\node[main node] (4) [above right of=3] {};

\node[main node] (5) [below right of=4] {};
\node[main node] (6) [right of=5] {};

\node[main node] (7) [below right of=1, xshift=-0.2cm,yshift=0.4cm] {};
\node[main node] (8) [below right of=2, xshift=-0.2cm,yshift=0.4cm] {};
\node[main node] (9) [below right of=5, xshift=-0.2cm,yshift=0.4cm] {};

\node[main node] (10) [above of=4,yshift=-0.3cm] {};

\node [right of=3, xshift=-0.3cm] {$\cdots$};

\node [above of=4, yshift=-0.7cm,xshift=0.1cm] {{\small $\ell$}};
\node [above of=10, yshift=-0.7cm,xshift=0.1cm] {{\small $\ell'$}};
\node [above of=6, yshift=-0.7cm,xshift=0.3cm] {{\small $\ell-1$}};
\node [above of=1, yshift=-0.7cm, xshift=-0.1cm] {{\small $1$}};
\node [below of=2, yshift=0.75cm] {{\small $2$}};
\node [below of=7, xshift=0.2cm, yshift=0.8cm] {{\small $1'$}};
\node [below of=8, xshift=0.2cm, yshift=0.8cm] {{\small $2'$}};

\path
(1) edge [->, >=stealth', shorten >=1pt, bend left=-60] node  {} (6);

\path
(4) edge (1)
(4) edge (2)
(4) edge (3)
(4) edge (5)
(4) edge (6);
\path
(7) edge (1)
(7) edge (2);
\path
(8) edge (2)
(8) edge (3);
\path
(9) edge (5)
(9) edge (6);
\path
(10) edge (1)
(10) edge (6);

\path
(1) [->,>=stealth', shorten >=1pt] edge (2);
\path
(2) [->,>=stealth', shorten >=1pt] edge (3);
\path
(5) [->,>=stealth', shorten >=1pt] edge (6);

\end{tikzpicture}

& 

\begin{tikzpicture}[node distance=1cm,auto,main node/.style={fill,circle,draw,inner sep=0pt,minimum size=5pt}]

\node[main node] (1) {};
\node[main node] (2) [right of=1] {};
\node[main node] (3) [right of=2] {};

\node[main node] (4) [above right of=3] {};

\node[main node] (5) [below right of=4] {};
\node[main node] (6) [right of=5] {};

\node[main node] (7) [below right of=1, xshift=-0.2cm,yshift=0.4cm] {};
\node[main node] (8) [below right of=2, xshift=-0.2cm,yshift=0.4cm] {};
\node[main node] (9) [below right of=5, xshift=-0.2cm,yshift=0.4cm] {};

\node[main node] (10) [above of=4,yshift=-0.3cm] {};

\node [right of=3, xshift=-0.3cm] {$\cdots$};

\node [above of=4, yshift=-0.7cm,xshift=0.1cm] {{\small $\ell$}};
\node [above of=10, yshift=-0.7cm,xshift=0.1cm] {{\small $\ell'$}};
\node [above of=6, yshift=-0.7cm,xshift=0.3cm] {{\small $\ell-1$}};
\node [above of=1, yshift=-0.7cm, xshift=-0.1cm] {{\small $1$}};
\node [below of=2, yshift=0.75cm] {{\small $2$}};
\node [below of=7, xshift=0.2cm, yshift=0.8cm] {{\small $1'$}};
\node [below of=8, xshift=0.2cm, yshift=0.8cm] {{\small $2'$}};

\path
(1) edge [->, >=stealth', shorten >=1pt, bend left=-60] node  {} (6);

\path
(4) edge (1)
(4) edge (2)
(4) edge (3)
(4) edge (5)
(4) edge (6);
\path
(7) edge (1)
(7) edge (2);
\path
(8) edge (2)
(8) edge (3);
\path
(9) edge (5)
(9) edge (6);
\path
(10) edge (1)
(10) edge (6);

\path
(1) [->,>=stealth', shorten >=1pt] edge (2);
\path
(2) [->,>=stealth', shorten >=1pt] edge (3);
\path
(5) [->,>=stealth', shorten >=1pt] edge (6);
\path
(10) [->,>=stealth', shorten >=1pt] edge (6);
\path
(1) [->,>=stealth', shorten >=1pt] edge (10);

\end{tikzpicture}

\end{tabular}
\end{center}
\vspace{-5mm}
\caption{Non-word-representability of $A_{\ell}$}\label{fig-nwr-Aell}
\end{figure}

Now, if $(\ell-1)\rightarrow \ell'$ were an edge, then the edge $1\ell'$ would either be a shortcutting edge (e.g. $2\rightarrow \ell'$ is missing), or would form a cycle taking into account the directed path $1\rightarrow 2\rightarrow \cdots \rightarrow (\ell-1)$. Thus, we must have $\ell' \rightarrow (\ell-1)$, and not to have a shortcut, we must also have $1\rightarrow \ell'$, as shown in the graph to the right in Figure~\ref{fig-nwr-Aell}.

Next, consider the triangle $121'$. Orienting it as $2\rightarrow 1'$ and $1'\rightarrow 1$ gives a cycle, while orienting it as $1\rightarrow 1'$ and $1'\rightarrow 2$ gives a shortcut induced by the vertices $1$, $1'$, $2$ and $3$ with the shortcutting edge $1'\rightarrow 3$.  
 On the other hand, similarly to the proof of Theorem~\ref{thm-K-Tri-m-w-r}, one can see that none of the orientations $1\rightarrow 1'$ and $2\rightarrow 1'$, or $1'\rightarrow 1$ and $1'\rightarrow 2$, results in a shortcut or a cycle. Similarly, no matter which of these orientations is selected, when considering the graph induced by the vertices 1, 2, $1'$ and $\ell$, we see that the orientation of the edges $1\ell$ and $2\ell$ must either be $1\rightarrow \ell$ and $2\rightarrow \ell$, or $\ell\rightarrow 1$ and $\ell\rightarrow 2$. 
 
Similar arguments as above can be applied to the graphs induced by $i$, $i'$, $(i+1)$ and $\ell$ for $i=2$, then $i=3$, etc, up to $i=\ell-2$, except for now the orientations of the edges $i\ell$ and $(i+1)\ell$ will be uniquely defined based on the orientation of the edge $1\ell$. Thus, we see that $\ell$ must either be a sink, or a source. Considering the graph induced by the vertices $1$, $(\ell-1)$, $\ell$ and $\ell'$ we see that in the former case, $1\rightarrow \ell$ is a shortcutting edge, while in the later case $\ell\rightarrow (\ell-1)$ is a shortcutting edge; contradiction. Thus, $A_{\ell}$ does not admit a semi-transitive orientation and thus is not word-representable. 
\end{proof}

\section{Our characterization results}

\subsection{Restricting degrees in $E_{n-m}$ to be at most 2}

\begin{definition} For a split graph $(E_{n-m},K_m)$, any triangle induced by two vertices in $K_m$ and one vertex in $E_{n-m}$ is called a {\em non-clique triangle}.\end{definition}

\begin{theorem}\label{main-1} Let $m\geq 1$ and $S_n=(E_{n-m},K_m)$ be a split graph. Also, let the degree of any vertex in $E_{n-m}$ be at most $2$. Then $S_n$ is word-representable if and only if $S_n$ does not contain the graphs $T_2$ in Figure~\ref{nonRepTri} and $A_{\ell}$ in Definition~\ref{def-Aell} as induced subgraphs. \end{theorem}

\begin{proof} By Lemma~\ref{lemma-assumptions}, we can assume that each vertex in $E_{n-m}$ is of degree 2, and no vertex in $E_{n-m}$ has the same neighbourhood. Because $T_2$ in Figure~\ref{nonRepTri} is non-word-representable, we see that no three non-clique triangles can be incident to the same vertex.   Moreover, by Theorem~\ref{Aell-min-non-repres}, we see that $K_m$ cannot have a cycle (without repeated vertices) such that each edge in the cycle is an edge in a non-clique triangle. These observations imply that $K_m$ contains disjoint paths, such that each edge in a path is an edge in a non-clique triangle, as shown schematically in Figure~\ref{schem-Km-tri}. But then we can redraw the graph, if necessary, to see that $S_n$ is exactly the graph $K^{\triangle}_\ell$ with possibly some of non-clique triangles missing, and this graph is word-representable by Theorem~\ref{thm-K-Tri-m-w-r} taking into account the hereditary nature of word-representability. 
\end{proof}

\begin{figure}[h]
\begin{center}
\begin{tabular}{cc}
\begin{tikzpicture}[node distance=0.5cm,auto,main node/.style={fill,circle,draw,inner sep=0pt,minimum size=3pt}]

\node[main node] (1) {};
\node[main node] (2) [above right of=1] {};
\node[main node] (3) [below right of=2] {};
\node[main node] (4) [above right of=3] {};
\node[main node] (5) [below right of=4] {};
\node[main node] (6) [above right of=5] {};
\node[main node] (7) [below right of=6] {};

\node (16) [above right of=2] {};
\node[main node] (17) [above left of=16] {};
\node[main node] (18) [above right of=17] {};
\node[main node] (19) [below right of=18] {};
\node[main node] (20) [above right of=19] {};
\node[main node] (21) [below right of=20] {};
\node[main node] (22) [above right of=21] {};
\node[main node] (23) [below right of=22] {};

\path
(17) edge (23)
(17) edge (18);
\path
(19) edge (18)
(19) edge (20);
\path
(21) edge (20)
(21) edge (22);

\path
(23) edge (22);

\node[main node] (8) [below of=3] {};
\node[main node] (9) [below left of=8] {};
\node[main node] (10) [below right of=8] {};
\node[main node] (11) [above right of=10] {};
\node[main node] (12) [below right of=11] {};

\node[main node] (13) [right of=12] {};
\node[main node] (14) [above right of=13] {};
\node[main node] (15) [below right of=14] {};

\node[main node] (24) [below right of=10] {};
\node[main node] (25) [below right of=24] {};
\node[main node] (26) [below left of=24] {};

\path
(25) edge (26)
(24) edge (25)
(24) edge (26);

\path
(9) edge (12)
(9) edge (8);
\path
(10) edge (8)
(10) edge (11);
\path
(11) edge (12);
\path
(1) edge (2)
(1) edge (7);
\path
(3) edge (2)
(3) edge (4);
\path
(5) edge (4)
(5) edge (6);
\path
(7) edge (6);
\path
(13) edge (15);
\path
(14) edge (13)
(14) edge (15);

\draw (1.5,0) circle (2cm); 

\end{tikzpicture}

\end{tabular}
\end{center}
\vspace{-5mm}
\caption{Schematic structure of the graph $S_n$ in Theorem~\ref{main-1}}\label{schem-Km-tri}
\end{figure}
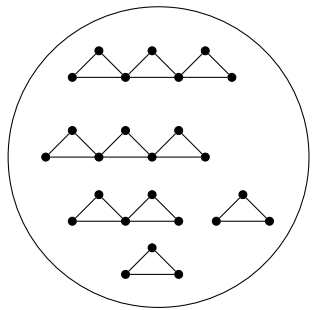

\subsection{Cliques of size 4}

We restrict our attention to the case of cliques of size 4 ($m=4$). If the degrees of vertices in $E_{n-4}$ are at most 2, we can apply Theorem~\ref{main-1} to see that word-representability is characterized by avoidance of the graphs $T_1$ and $T_2$ in Figure~\ref{nonRepTri} as induced subgraphs. However, $E_{n-4}$ may also have vertices of degree 3. Theorem~\ref{main-2} below gives a complete characterization for word-representability of $(E_{n-4},K_4)$.

Our methodology to prove Theorem~\ref{main-2} is in using Lemma~\ref{lemma-assumptions} to come up with the largest possible split graph $S_n$ in the context. We then identify a minimal non-word-representable induced subgraph in such $S_n$ and consider a smaller graph $S_{n-1}$ obtained from $S_n$ by removing one vertex. We need to consider all possibilities of removing a vertex in $S_n$, but  we use symmetries, whenever possible, to reduce the number of cases to consider. If $S_{n-1}$ is word-representable, there is nothing to do. Otherwise, we repeat the process for $S_{n-1}$ instead of $S_n$. This way, we located all minimal non-word-representable induced subgraphs.  We note that in the proof, orientations claimed by us to be semi-transitive, can be checked to be such either by hand, or using the software~\cite{G}.

\begin{lemma} The split graph $T_4$ in Figure~\ref{nonRepTri-2} is minimal non-word-representable graph. \end{lemma}

\begin{proof} Non-word-representability of $T_4$ can be proved rigorously by the branching method explained in Section 4.5 in \cite{KL15} (also, see e.g. \cite{CKL17} where the method is applied). However, recording such a proof would take up to two pages while brining no insights, so to save space, we simply refer to the software \cite{G} justifying non-word-representability of $T_4$.  

The minimality of $T_4$ follows from the fact that removing a vertex in $T_4$ we do not obtain one of the graphs in Figure 3.9 on page 48 in \cite{KL15} showing all 25 non-word-representable graphs on 7 vertices. Alternatively, one can use the software \cite{G}, or follow the cases in the proof of Theorem~\ref{main-2} (for example, removing a vertex of degree 2 in $T_4$ is equivalent to removing vertex 8 in the semi-transitively oriented graph $M_6$ in Figure~\ref{maximal-3-2-config-2}, which ought to result in a word-representable graph). \end{proof}

\begin{theorem}\label{main-2} Let $S_n=(E_{n-4},K_4)$ be a split graph. Then  $S_n$ is word-representable if and only if $S_n$ does not contain the graphs $T_1$, $T_2$ and $T_3$ in Figure~\ref{nonRepTri}, and $T_4$ in Figure~\ref{nonRepTri-2} as induced subgraphs.\end{theorem}

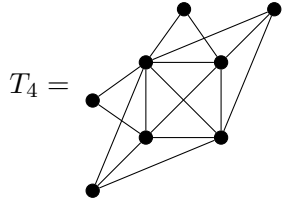
\begin{figure}
\begin{center}

\begin{tikzpicture}[node distance=1cm,auto,main node/.style={fill,circle,draw,inner sep=0pt,minimum size=5pt}]

\node[main node] (1) {};
\node[main node] (2) [right of=1] {};
\node[main node] (3) [below of=2] {};
\node[main node] (4) [left of=3] {};
\node[main node] (5) [above right of=1,xshift=-0.2cm] {};
\node[main node] (6) [above right of=2] {};
\node[main node] (7) [below left of=1,yshift=0.2cm] {};
\node[main node] (8) [below left of=4] {};

\node [above left of=7,yshift=-0.5cm] {$T_4=$};

\path
(1) edge (2)
(1) edge (3)
(1) edge (4)
(1) edge (5)
(1) edge (6)
(1) edge (7)
(1) edge (8);

\path
(2) edge (3)
(2) edge (4)
(2) edge (5)
(2) edge (6);

\path
(3) edge (4)
(3) edge (6)
(3) edge (8);

\path
(4) edge (8)
(4) edge (7);

\end{tikzpicture}

\caption{\label{nonRepTri-2} A minimal non-word-representable split graph $T_4$}
\end{center}
\end{figure}

\begin{proof} We can assume that $E_{n-4}$ contains at least one vertex of degree 3, or else we are done by Theorem~\ref{main-1} with $T_1=A_3$ and $T_2$ being forbidden induced subgraphs. Further, recall that by Lemma~\ref{lemma-assumptions}, we can assume that each vertex in $E_{n-m}$ is of degree 2 or 3, and no vertices in $E_{n-m}$ have the same neighbourhood.  

\begin{figure}
\begin{center}

\begin{tabular}{ccc}

\begin{tikzpicture}[node distance=1cm,auto,main node/.style={fill,circle,draw,inner sep=0pt,minimum size=5pt}]

\node[main node] (1) {};
\node[main node] (2) [right of=1] {};
\node[main node] (3) [below of=2] {};
\node[main node] (4) [left of=3] {};
\node[main node] (5) [above right of=2] {};
\node[main node] (6) [below left of=4] {};

\node [left of=1,xshift=0.8cm,yshift=0.2cm] {{\small 1}};
\node [right of=2,xshift=-0.7cm,yshift=0.1cm] {{\small 2}};
\node [right of=3,xshift=-0.7cm] {{\small 3}};
\node [left of=4,xshift=0.9cm,yshift=0.3cm] {{\small 4}};
\node [right of=5,xshift=-0.7cm] {{\small 5}};
\node [left of=6,xshift=0.7cm,yshift=-0.4] {{\small 6}};

\path
(1) edge (2)
(1) edge (3)
(1) edge (4)
(1) edge (5)
(1) edge (6);

\path
(2) edge (3)
(2) edge (4)
(2) edge (5);

\path
(3) edge (6)
(3) edge (4)
(3) edge (5);

\path
(4) edge (6);

\end{tikzpicture}

&

&

\begin{tikzpicture}[node distance=1cm,auto,main node/.style={fill,circle,draw,inner sep=0pt,minimum size=5pt}]

\node[main node] (1) {};
\node[main node] (2) [right of=1] {};
\node[main node] (3) [below of=2] {};
\node[main node] (4) [left of=3] {};
\node[main node] (5) [above right of=2] {};
\node[main node] (6) [below right of=3] {};

\node [left of=1,xshift=0.8cm,yshift=0.2cm] {{\small 2}};
\node [right of=2,xshift=-0.7cm,yshift=0.1cm] {{\small 1}};
\node [right of=3,xshift=-0.7cm] {{\small 3}};
\node [left of=4,xshift=0.8cm,yshift=0.3cm] {{\small 4}};
\node [right of=5,xshift=-0.7cm] {{\small 5}};
\node [right of=6,xshift=-0.7cm,yshift=-0.4] {{\small 6}};

\node [left=1,yshift=-0.5cm,xshift=-0.5cm] {=};

\path
(1) edge (2)
(1) edge (3)
(1) edge (4)
(1) edge (5);

\path
(2) edge (3)
(2) edge (4)
(2) edge (5)
(2) edge (6);

\path
(3) edge (6)
(3) edge (4)
(3) edge (5);

\path
(4) edge (6);

\end{tikzpicture}

\end{tabular}

\caption{\label{maximal-3-configurations} The single maximal possibility, up to isomorphism, for $(E_{n-4},K_4)$ with vertices of degree 3 in $E_{n-4}$}
\end{center}
\end{figure}
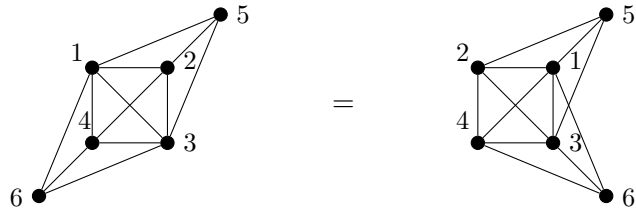

Assuming that $E_{n-4}$ only contains vertices of degree 3, we can see that $T_1$ and $T_3$ are the only minimal non-word-representable induced subgraphs to be avoided by $S_n$ to be word-representable. Indeed, since no vertices in $E_{n-m}$ have the same neighbourhood, $E_{n-m}$ can have at most 4 vertices in this case. If all 4 vertices are present, $S_n$ contains the minimal non-word-representable $T_3$ as an induced subgraph (to see this, $T_3$ is redrawn in a different way in Figure~\ref{nonRepTri}). Removing one of the 4 vertices in $E_{n-m}$ (any one due to the symmetries) we obtain exactly $T_3$ which is a minimal non-word-representable graph. It remains to notice that if $E_{n-4}$ contains 4 vertices and we will remove a vertex in $K_4$ we will obtain the minimal non-word-representable graph $T_1$. Thus, $E_{n-4}$ can have at most two vertices in this case, resulting, up to isomorphism, in a single case to consider  that is presented in Figure~\ref{maximal-3-configurations} (along with a justification that two of seemly different graphs are actually isomorphic).

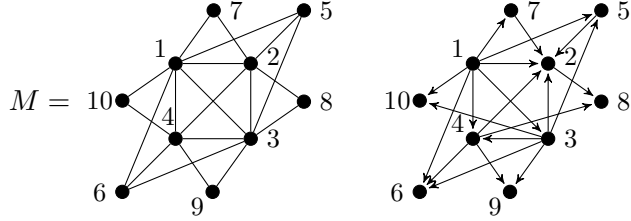
\begin{figure}
\begin{center}

\begin{tabular}{cc}

\begin{tikzpicture}[node distance=1cm,auto,main node/.style={fill,circle,draw,inner sep=0pt,minimum size=5pt}]

\node[main node] (1) {};
\node[main node] (2) [right of=1] {};
\node[main node] (3) [below of=2] {};
\node[main node] (4) [left of=3] {};
\node[main node] (5) [above right of=2] {};
\node[main node] (6) [below left of=4] {};
\node[main node] (7) [above right of=1, xshift=-0.2cm] {};
\node[main node] (8) [below right of=2,yshift=0.2cm] {};
\node[main node] (9) [below left of=3, xshift=0.2cm] {};
\node[main node] (10) [above left of=4,yshift=-0.2cm] {};

\node [left of=1,xshift=0.8cm,yshift=0.2cm] {{\small 1}};
\node [right of=2,xshift=-0.7cm,yshift=0.1cm] {{\small 2}};
\node [right of=3,xshift=-0.7cm] {{\small 3}};
\node [left of=4,xshift=0.9cm,yshift=0.3cm] {{\small 4}};
\node [right of=5,xshift=-0.7cm] {{\small 5}};
\node [left of=6,xshift=0.7cm,yshift=-0.4] {{\small 6}};
\node [right of=7,xshift=-0.7cm] {{\small 7}};
\node [right of=8,xshift=-0.7cm] {{\small 8}};
\node [left of=9,xshift=0.8cm,yshift=-0.2cm] {{\small 9}};
\node (11) [left of=10,xshift=0.7cm] {{\small 10}};
\node [left of=11,xshift=0.2cm] {$M=$};

\path
(10) edge (1)
(10) edge (4)
(9) edge (3)
(9) edge (4)
(8) edge (2)
(8) edge (3)
(7) edge (1)
(7) edge (2);

\path
(1) edge (2)
(1) edge (3)
(1) edge (4)
(1) edge (5)
(1) edge (6);

\path
(2) edge (3)
(2) edge (4)
(2) edge (5);

\path
(3) edge (6)
(3) edge (4)
(3) edge (5);

\path
(4) edge (6);

\end{tikzpicture}

&

\begin{tikzpicture}[->,>=stealth', shorten >=1pt, node distance=1cm,auto,main node/.style={fill,circle,draw,inner sep=0pt,minimum size=5pt}]

\node[main node] (1) {};
\node[main node] (2) [right of=1] {};
\node[main node] (3) [below of=2] {};
\node[main node] (4) [left of=3] {};
\node[main node] (5) [above right of=2] {};
\node[main node] (6) [below left of=4] {};
\node[main node] (7) [above right of=1, xshift=-0.2cm] {};
\node[main node] (8) [below right of=2,yshift=0.2cm] {};
\node[main node] (9) [below left of=3, xshift=0.2cm] {};
\node[main node] (10) [above left of=4,yshift=-0.2cm] {};

\node [left of=1,xshift=0.8cm,yshift=0.2cm] {{\small 1}};
\node [right of=2,xshift=-0.7cm,yshift=0.1cm] {{\small 2}};
\node [right of=3,xshift=-0.7cm] {{\small 3}};
\node [left of=4,xshift=0.8cm,yshift=0.2cm] {{\small 4}};
\node [right of=5,xshift=-0.7cm] {{\small 5}};
\node [left of=6,xshift=0.7cm,yshift=-0.4] {{\small 6}};
\node [right of=7,xshift=-0.7cm] {{\small 7}};
\node [right of=8,xshift=-0.7cm] {{\small 8}};
\node [left of=9,xshift=0.8cm,yshift=-0.2cm] {{\small 9}};
\node [left of=10,xshift=0.7cm] {{\small 10}};

\path
(1) edge (2)
(1) edge (3)
(1) edge (4)
(1) edge (5)
(1) edge (6)
(1) edge (7)
(1) edge (10);

\path
(2) edge (8);

\path
(3) edge (2)
(3) edge (4)
(3) edge (5)
(3) edge (6)
(3) edge (9)
(3) edge (10);

\path
(4) edge (2)
(4) edge (6)
(4) edge (8)
(4) edge (9);

\path
(5) edge (2);

\path
(7) edge (2);

\end{tikzpicture}

\end{tabular}

\caption{\label{maximal-3-2-config} Maximal non-isomorphic possibilities for $(E_{n-4},K_4)$ with vertices of degree 2 and 3 in $E_{n-4}$}
\end{center}
\end{figure}

We next consider adding vertices of degree 2 to $E_{n-4}$ in the graph in Figure~\ref{maximal-3-configurations}.  As mentioned above, non-word-representability of  $T_2$ in Figure~\ref{nonRepTri} implies that no three non-clique triangles (with disjoint vertices) can be incident to the same vertex. Thus, at most four vertices (with distinct neighbourhoods) of degree 2 can be present in $E_{n-4}$, and there are just two non-isomorphic ways to add these vertices to the graph in Figure~\ref{maximal-3-configurations} that are given in Figure~\ref{maximal-3-2-config}. 

The graph to the right in Figure~\ref{maximal-3-2-config} is word-representable and we provide one of its semi-transitive orientations to justify this (we omit a justification that the orientation is semi-transitive). On the other hand, the graph $M$ in Figure~\ref{maximal-3-2-config} is non-word-representable because it contains the graph $T_4$ in Figure~\ref{nonRepTri-2} as an induced subgraph (just remove the vertices 7 and 10 to see this). 

To find all minimal non-word-representable induced subgraphs in $M$, we will consider removing one vertex from it. Note that there are only four cases to consider up to isomorphism.

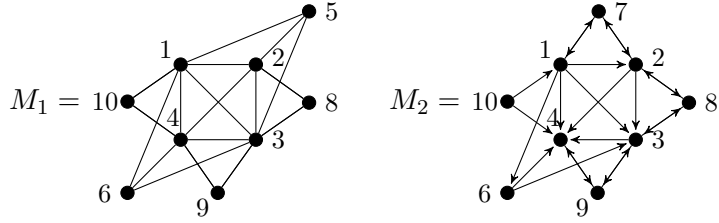
\begin{figure}
\begin{center}

\begin{tabular}{cc}

\begin{tikzpicture}[node distance=1cm,auto,main node/.style={fill,circle,draw,inner sep=0pt,minimum size=5pt}]

\node[main node] (1) {};
\node[main node] (2) [right of=1] {};
\node[main node] (3) [below of=2] {};
\node[main node] (4) [left of=3] {};
\node[main node] (5) [above right of=2] {};
\node[main node] (6) [below left of=4] {};
\node[main node] (8) [below right of=2,yshift=0.2cm] {};
\node[main node] (9) [below left of=3, xshift=0.2cm] {};
\node[main node] (10) [above left of=4,yshift=-0.2cm] {};

\node [left of=1,xshift=0.8cm,yshift=0.2cm] {{\small 1}};
\node [right of=2,xshift=-0.7cm,yshift=0.1cm] {{\small 2}};
\node [right of=3,xshift=-0.7cm] {{\small 3}};
\node [left of=4,xshift=0.9cm,yshift=0.3cm] {{\small 4}};
\node [right of=5,xshift=-0.7cm] {{\small 5}};
\node [left of=6,xshift=0.7cm,yshift=-0.4] {{\small 6}};
\node [right of=8,xshift=-0.7cm] {{\small 8}};
\node [left of=9,xshift=0.8cm,yshift=-0.2cm] {{\small 9}};
\node (11) [left of=10,xshift=0.7cm] {{\small 10}};

\node [left of=11,xshift=0.2cm] {$M_1=$};

\path
(10) edge (1)
(10) edge (4)
(9) edge (3)
(9) edge (4)
(6) edge (1)
(6) edge (3)
(8) edge (2)
(8) edge (3);

\path
(1) edge (2)
(1) edge (3)
(1) edge (4)
(1) edge (5)
(1) edge (10);

\path
(2) edge (3)
(2) edge (4)
(2) edge (5)
(2) edge (8);

\path
(3) edge (9)
(3) edge (8)
(3) edge (4)
(3) edge (5);

\path
(4) edge (6)
(4) edge (9);

\path
(10) edge (4);

\end{tikzpicture}

&

\begin{tikzpicture}[->,>=stealth', shorten >=1pt, node distance=1cm,auto,main node/.style={fill,circle,draw,inner sep=0pt,minimum size=5pt}]

\node[main node] (1) {};
\node[main node] (2) [right of=1] {};
\node[main node] (3) [below of=2] {};
\node[main node] (4) [left of=3] {};
\node[main node] (6) [below left of=4] {};
\node[main node] (7) [above right of=1, xshift=-0.2cm] {};
\node[main node] (8) [below right of=2,yshift=0.2cm] {};
\node[main node] (9) [below left of=3, xshift=0.2cm] {};
\node[main node] (10) [above left of=4,yshift=-0.2cm] {};

\node [left of=1,xshift=0.8cm,yshift=0.2cm] {{\small 1}};
\node [right of=2,xshift=-0.7cm,yshift=0.1cm] {{\small 2}};
\node [right of=3,xshift=-0.7cm] {{\small 3}};
\node [left of=4,xshift=0.9cm,yshift=0.3cm] {{\small 4}};
\node [left of=6,xshift=0.7cm,yshift=-0.4] {{\small 6}};
\node [right of=7,xshift=-0.7cm] {{\small 7}};
\node [right of=8,xshift=-0.7cm] {{\small 8}};
\node [left of=9,xshift=0.8cm,yshift=-0.2cm] {{\small 9}};
\node (11) [left of=10,xshift=0.7cm] {{\small 10}};
\node [left of=11,xshift=0.2cm] {$M_2=$};

\path
(10) edge (1)
(10) edge (4)
(9) edge (3)
(9) edge (4)
(8) edge (2)
(8) edge (3)
(7) edge (1)
(7) edge (2);

\path
(1) edge (2)
(1) edge (3)
(1) edge (4)
(1) edge (6)
(1) edge (7);

\path
(2) edge (3)
(2) edge (4)
(2) edge (7)
(2) edge (8);

\path
(3) edge (8)
(3) edge (4)
(3) edge (9);

\path
(4) edge (9);

\path
(6) edge (3)
(6) edge (4);

\end{tikzpicture}

\end{tabular}

\caption{\label{maximal-3-2-config-1} Two cases to consider in the proof of Theorem~\ref{main-2}}
\end{center}
\end{figure}

\begin{itemize}
\item If vertex 1 is removed in $M$, then vertices 7 and 10 will be of degree 1 and can also be removed by Lemma~\ref{lemma-assumptions}. Moreover, the vertices 6 and 9 will have the same neighbourhoods, and by the same lemma, one of these vertices can be removed. The same applies to vertices 5 and 8, resulting in a graph on 5 vertices induced by, say, vertices 2, 3, 4, 5, 6, and any graph on 5 vertices is word-representable.  
\item If the vertex 2 is removed in $M$, then vertices 7 and 8 will be of degree 1 and thus can also be removed by Lemma~\ref{lemma-assumptions}. This leaves us with a graph on 6 vertices which is word-representable because it is different from $W_5$, the only non-word-representable graph on 6 vertices. 
\item If vertex 5 is removed, then we obtain the graph $M_2$ in Figure~\ref{maximal-3-2-config-1}, which is word-representable because of the semi-transitive orientation we provide in the figure (we omit a justification that the orientation is semi-transitive).
\item Finally, if vertex 3 is removed in $M$, we will obtain the non-word-representable graph $M_1$ in Figure~\ref{maximal-3-2-config-1} (it contains $T_4$).
\end{itemize}
To complete our proof, we need to remove a vertex in $M_1$. Unfortunately, no symmetries can be applied here, so we have to consider 9 cases.

\begin{itemize}
\item If vertex 1 is removed then vertex 10 will be of degree 1 and it can be removed by Lemma~\ref{lemma-assumptions}. The resulting graph is word-representable because it is clearly a subgraph of $T_4$, and $T_4$ is a minimal non-word-representable. 
\item If vertex 2 is removed then vertex 8 will be of degree 1 and it can be removed by Lemma~\ref{lemma-assumptions}. The resulting graph is precisely the non-word-representable graph $T_1$.
\item If vertex 3 is removed then vertices 8 and 9 will be of degree 1 and they can be removed by Lemma~\ref{lemma-assumptions}. The resulting graph is on 6 vertices, it is not $W_5$ and thus is word-representable. 
\item If vertex 4 is removed then vertices 9 and 10 will be of degree 1 and they can be removed by Lemma~\ref{lemma-assumptions}. The resulting graph is on 6 vertices, it is not $W_5$ and thus is word-representable. 
\item If vertex 5 is removed then we will obtain a word-representable graph $M_3$ in Figure~\ref{maximal-3-2-config-2}, where we provide a semi-transitive orientation of the graph without justification. 
\item If vertex 6 is removed then we will obtain a word-representable graph $M_4$ in Figure~\ref{maximal-3-2-config-2}, where we provide a semi-transitive orientation of the graph without justification. 
\item If vertex 8 is removed then we will obtain the graph $M_5$ in Figure~\ref{maximal-3-2-config-2}. This graph contains $T_1$ as an induced subgraph (remove vertex 2 to see it).To complete this case, we need to remove a vertex in $M_5$ other than vertex 2, to make sure that a word-representable graph would be obtained. 

\begin{itemize}
\item Removing vertex 1, which is clearly equivalent to removing vertex 3, gives vertex 8 of degree 1 which can be removed by Lemma~\ref{lemma-assumptions}. Moreover, one of vertices 6 and 9 can be removed by Lemma~\ref{lemma-assumptions} because they have the same neighbourhood. This results in a graph on 5 vertices, but any such graph is word-representable.  
\item Removing vertex 4 gives two vertices, 9 and 10, that can be removed by Lemma~\ref{lemma-assumptions}. The resulting graph is on 5 vertices and it must word-representable. 
\item Removing vertex 5 is equivalent to removing vertices 7 and 8 in the graph $M_2$ in Figure~\ref{maximal-3-2-config-1}, so this graph is word-representable.
\item Removing vertex 6 is equivalent to removing vertices 9 and 10 in the graph $M_2$ in Figure~\ref{maximal-3-2-config-1}
\item Finally, removing vertex 9, which is clearly equivalent to removing vertex 10, gives the graph obtained from the semi-transitively oriented graph $M_6$ in Figure~\ref{maximal-3-2-config-2}, and it is word-representable. 
\end{itemize}

\item If vertex 9 is removed then the semi-transitively oriented graph $M_6$ in Figure~\ref{maximal-3-2-config-2} is obtained (we omit justification that the orientation is indeed semi-transitive).
\item Finally, if vertex 10 is removed then we will obtain the minimal non-word-representable graph $T_4$.
\end{itemize}

\begin{figure}
\begin{center}

\begin{tabular}{cc}

\begin{tikzpicture}[->,>=stealth', shorten >=1pt,node distance=1cm,auto,main node/.style={fill,circle,draw,inner sep=0pt,minimum size=5pt}]

\node[main node] (1) {};
\node[main node] (2) [right of=1] {};
\node[main node] (3) [below of=2] {};
\node[main node] (4) [left of=3] {};
\node[main node] (6) [below left of=4] {};
\node[main node] (8) [below right of=2,yshift=0.2cm] {};
\node[main node] (9) [below left of=3, xshift=0.2cm] {};
\node[main node] (10) [above left of=4,yshift=-0.2cm] {};

\node [left of=1,xshift=0.8cm,yshift=0.2cm] {{\small 1}};
\node [right of=2,xshift=-0.7cm,yshift=0.1cm] {{\small 2}};
\node [right of=3,xshift=-0.7cm] {{\small 3}};
\node [left of=4,xshift=0.9cm,yshift=0.3cm] {{\small 4}};
\node [left of=6,xshift=0.7cm,yshift=-0.4] {{\small 6}};
\node [right of=8,xshift=-0.7cm] {{\small 8}};
\node [left of=9,xshift=0.8cm,yshift=-0.2cm] {{\small 9}};
\node (11) [left of=10,xshift=0.7cm] {{\small 10}};

\node [left of=11,xshift=0.2cm] {$M_3=$};

\path
(1) edge (2)
(1) edge (3)
(1) edge (4)
(1) edge (6)
(1) edge (10);

\path
(2) edge (3)
(2) edge (4)
(2) edge (8);

\path
(3) edge (9)
(3) edge (8)
(3) edge (4);

\path
(4) edge (9);

\path
(6) edge (3)
(6) edge (4);

\path
(10) edge (4);

\end{tikzpicture}

&

\begin{tikzpicture}[->,>=stealth', shorten >=1pt, node distance=1cm,auto,main node/.style={fill,circle,draw,inner sep=0pt,minimum size=5pt}]

\node[main node] (1) {};
\node[main node] (2) [right of=1] {};
\node[main node] (3) [below of=2] {};
\node[main node] (4) [left of=3] {};
\node[main node] (5) [above right of=2] {};
\node[main node] (8) [below right of=2,yshift=0.2cm] {};
\node[main node] (9) [below left of=3, xshift=0.2cm] {};
\node[main node] (10) [above left of=4,yshift=-0.2cm] {};

\node [left of=1,xshift=0.8cm,yshift=0.2cm] {{\small 1}};
\node [right of=2,xshift=-0.7cm,yshift=0.1cm] {{\small 2}};
\node [right of=3,xshift=-0.7cm] {{\small 3}};
\node [left of=4,xshift=0.9cm,yshift=0.3cm] {{\small 4}};
\node [right of=5,xshift=-0.7cm] {{\small 5}};
\node [right of=8,xshift=-0.7cm] {{\small 8}};
\node [left of=9,xshift=0.8cm,yshift=-0.2cm] {{\small 9}};
\node (11) [left of=10,xshift=0.7cm] {{\small 10}};

\node [left of=11,xshift=0.2cm] {$M_4=$};

\path
(1) edge (2)
(1) edge (3)
(1) edge (4)
(1) edge (5)
(1) edge (10);

\path
(2) edge (3)
(2) edge (4)
(2) edge (5)
(2) edge (8);

\path
(3) edge (9)
(3) edge (8)
(3) edge (4)
(3) edge (5);

\path
(4) edge (9);

\path
(10) edge (4);

\end{tikzpicture}

\end{tabular}

\begin{tabular}{cc}

\begin{tikzpicture}[node distance=1cm,auto,main node/.style={fill,circle,draw,inner sep=0pt,minimum size=5pt}]

\node[main node] (1) {};
\node[main node] (2) [right of=1] {};
\node[main node] (3) [below of=2] {};
\node[main node] (4) [left of=3] {};
\node[main node] (5) [above right of=2] {};
\node[main node] (6) [below left of=4] {};
\node[main node] (9) [below left of=3, xshift=0.2cm] {};
\node[main node] (10) [above left of=4,yshift=-0.2cm] {};

\node [left of=1,xshift=0.8cm,yshift=0.2cm] {{\small 1}};
\node [right of=2,xshift=-0.7cm,yshift=0.1cm] {{\small 2}};
\node [right of=3,xshift=-0.7cm] {{\small 3}};
\node [left of=4,xshift=0.9cm,yshift=0.3cm] {{\small 4}};
\node [right of=5,xshift=-0.7cm] {{\small 5}};
\node [left of=6,xshift=0.7cm,yshift=-0.4] {{\small 6}};
\node [left of=9,xshift=0.8cm,yshift=-0.2cm] {{\small 9}};
\node (11) [left of=10,xshift=0.7cm] {{\small 10}};

\node [left of=11,xshift=0.2cm] {$M_5=$};

\path
(10) edge (1)
(10) edge (4)
(9) edge (3)
(9) edge (4)
(6) edge (1)
(6) edge (3);

\path
(1) edge (2)
(1) edge (3)
(1) edge (4)
(1) edge (5)
(1) edge (10);

\path
(2) edge (3)
(2) edge (4)
(2) edge (5);

\path
(3) edge (9)
(3) edge (4)
(3) edge (5);

\path
(4) edge (6)
(4) edge (9);

\path
(10) edge (4);

\end{tikzpicture}

&

\begin{tikzpicture}[->,>=stealth', shorten >=1pt, node distance=1cm,auto,main node/.style={fill,circle,draw,inner sep=0pt,minimum size=5pt}]

\node[main node] (1) {};
\node[main node] (2) [right of=1] {};
\node[main node] (3) [below of=2] {};
\node[main node] (4) [left of=3] {};
\node[main node] (5) [above right of=2] {};
\node[main node] (6) [below left of=4] {};
\node[main node] (8) [below right of=2,yshift=0.2cm] {};
\node[main node] (10) [above left of=4,yshift=-0.2cm] {};

\node [left of=1,xshift=0.8cm,yshift=0.2cm] {{\small 1}};
\node [right of=2,xshift=-0.7cm,yshift=0.1cm] {{\small 2}};
\node [right of=3,xshift=-0.7cm] {{\small 3}};
\node [left of=4,xshift=0.9cm,yshift=0.3cm] {{\small 4}};
\node [right of=5,xshift=-0.7cm] {{\small 5}};
\node [left of=6,xshift=0.7cm,yshift=-0.4] {{\small 6}};
\node [right of=8,xshift=-0.7cm] {{\small 8}};
\node (11) [left of=10,xshift=0.7cm] {{\small 10}};

\node [left of=11,xshift=0.2cm] {$M_6=$};

\path
(1) edge (2)
(1) edge (3)
(1) edge (4)
(1) edge (5)
(1) edge (6)
(1) edge (10);

\path
(2) edge (3)
(2) edge (8);

\path
(3) edge (8);

\path
(4) edge (6)
(4) edge (3)
(4) edge (10)
(4) edge (2);

\path
(5) edge (2)
(5) edge (3);

\path
(6) edge (3);

\end{tikzpicture}

\end{tabular}

\caption{\label{maximal-3-2-config-2} Four subcases to consider in the proof of Theorem~\ref{main-2}}
\end{center}
\end{figure}
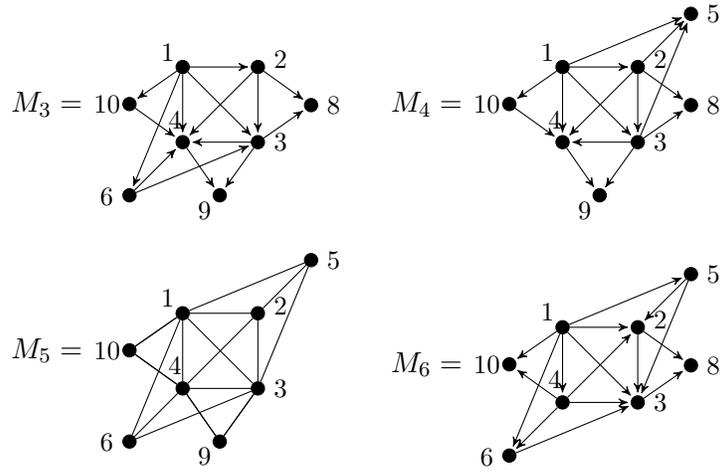

\end{proof}

\section{Semi-transitive orientations on split graphs}

Let $S_n=(E_{n-m},K_m)$ be a word-representable split graph. Then by Theorem~\ref{key-thm}, $S_n$ admits a semi-transitive orientation. Further, by Lemma~\ref{lem-tran-orie} we known that any such orientation induces a transitive orientation on $K_m$ that can be presented schematically as in Figure~\ref{schem-structure}, where we show the longest directed path in $K_m$, denoted by $\vec{P}$, but do not draw the other edges in $K_m$ even though they exist.

Theorems~\ref{semi-tran-groups} and~\ref{relative-order} below describe the structure of semi-transitive orientations in arbitrary word-representable split graph. 

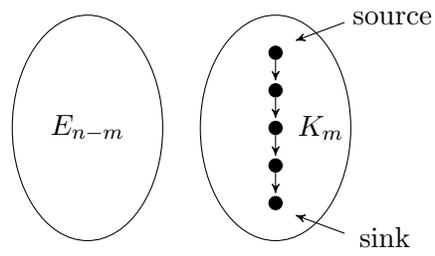
\begin{figure}
\begin{center}

\begin{tikzpicture}[->,>=stealth', shorten >=1pt,node distance=0.5cm,auto,main node/.style={fill,circle,draw,inner sep=0pt,minimum size=5pt}]

\node (1) {};
\node[main node] (2) [right of=1,xshift=5mm] {};
\node[main node] (3) [above of=2] {};
\node[main node] (4) [above of=3] {};
\node[main node] (5) [below of=2] {};
\node[main node] (6) [below of=5] {};

\node (7) [above of=4,yshift=-4mm,xshift=1mm] {};
\node (8) [above right of=7,xshift=6mm]{};
\node (9) [right of=8]{source};

\node (10) [below of=6,yshift=4mm,xshift=1mm] {};
\node (11) [below right of=10,xshift=6mm]{};
\node (12) [right of=11,xshift=-1mm]{sink};

\node [right of=2,xshift=1mm]{$K_m$};
\node [right of=2,xshift=-30mm]{$E_{n-m}$};

\path
(8) edge (7)
(11) edge (10);

\draw [] (-1.5,0) ellipse (10mm and 15mm); 
\draw [] (1,0) ellipse (10mm and 15mm); 

\path
(4) edge (3)
(3) edge (2)
(2) edge (5)
(5) edge (6);

\end{tikzpicture}

\caption{\label{schem-structure} A schematic structure of a semi-transitively oriented split graph}
\end{center}
\end{figure}

\begin{theorem}\label{semi-tran-groups} Any semi-transitive orientation of $S_n=(E_{n-m},K_m)$ subdivides the set of all vertices in $E_{n-m}$ into three, possibly empty, groups of the types shown schematically in Figure~\ref{3-groups}. In that figure,  
\begin{itemize}
\item similarly to Figure~\ref{schem-structure}, the vertical oriented paths are a schematic way to show (parts of) $\vec{P}$;
\item the vertical oriented paths in the types A and B represent up to $m$ consecutive vertices in $\vec{P}$; 
\item the vertical oriented path in the type C contains all $m$ vertices in $\vec{P}$, and it is subdivided into three groups of consecutive vertices, with the middle group possibly containing no vertices; the group of vertices containing the source (resp., sink) is the {\em source-group} (resp., {\em sink-group});
\item the vertices to the left of the vertical paths are from  $E_{n-m}$. 
\end{itemize} \end{theorem}

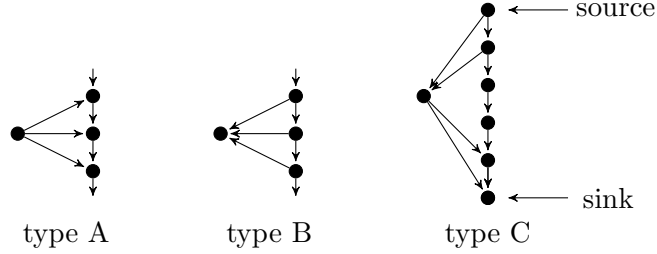
\begin{figure}
\begin{center}

\begin{tabular}{ccc}

\begin{tikzpicture}[->,>=stealth', shorten >=1pt,node distance=0.5cm,auto,main node/.style={fill,circle,draw,inner sep=0pt,minimum size=5pt}]

\node[main node] (1) {};
\node[main node] (2) [right of=1,xshift=5mm] {};
\node[main node] (3) [above of=2] {};
\node (4) [above of=3] {};
\node[main node] (5) [below of=2] {};
\node (6) [below of=5] {};

\node (7) [above of=4,yshift=-4mm,xshift=1mm] {};
\node (8) [above right of=7,xshift=6mm]{};

\node (10) [below of=6,yshift=4mm,xshift=1mm] {};
\node (11) [below right of=10,xshift=6mm]{};

\node (11) [below left of=5,yshift=-5mm]{type A};

\path
(4) edge (3)
(3) edge (2)
(2) edge (5)
(5) edge (6);

\path
(1) edge (2)
(1) edge (3)
(1) edge (5);

\end{tikzpicture}

&
\begin{tikzpicture}[->,>=stealth', shorten >=1pt,node distance=0.5cm,auto,main node/.style={fill,circle,draw,inner sep=0pt,minimum size=5pt}]

\node[main node] (1) {};
\node[main node] (2) [right of=1,xshift=5mm] {};
\node[main node] (3) [above of=2] {};
\node (4) [above of=3] {};
\node[main node] (5) [below of=2] {};
\node (6) [below of=5] {};

\node (7) [above of=4,yshift=-4mm,xshift=1mm] {};
\node (8) [above right of=7,xshift=6mm]{};

\node (10) [below of=6,yshift=4mm,xshift=1mm] {};
\node (11) [below right of=10,xshift=6mm]{};

\node (11) [below left of=5,yshift=-5mm]{type B};

\path
(4) edge (3)
(3) edge (2)
(2) edge (5)
(5) edge (6);

\path
(2) edge (1)
(3) edge (1)
(5) edge (1);

\end{tikzpicture}

&

\begin{tikzpicture}[->,>=stealth', shorten >=1pt,node distance=0.5cm,auto,main node/.style={fill,circle,draw,inner sep=0pt,minimum size=5pt}]

\node[main node] (1) {};
\node[main node] (2) [below right of=1,xshift=5mm] {};
\node[main node] (3) [above of=2] {};
\node[main node] (4) [above of=3] {};
\node[main node] (5) [below of=2] {};
\node[main node] (6) [below of=5] {};
\node[main node] (13) [above of=4] {};
\node[main node] (14) [below of=5] {};

\node (7) [above of=13,yshift=-5mm,xshift=1mm] {};
\node (8) [right of=7,xshift=6mm]{};
\node (9) [right of=8]{source};

\node (10) [below of=14,yshift=5mm,xshift=1mm] {};
\node (11) [right of=10,xshift=6mm]{};
\node (12) [right of=11,xshift=-1mm]{sink};

\node [below of=14]{type C};

\path
(8) edge (7)
(11) edge (10);

\path
(13) edge (1)
(4) edge (1)
(1) edge (14)
(1) edge (5)
(4) edge (3)
(3) edge (2)
(2) edge (5)
(5) edge (6)
(13) edge (4)
(5) edge (14);

\end{tikzpicture}

\end{tabular}

\caption{\label{3-groups} Three types of vertices in $E_{n-m}$ under a semi-transitive orientation of $(E_{n-m},K_m)$. The vertical oriented paths are a schematic way to show (parts of) $\vec{P}$}
\end{center}
\end{figure}

\begin{proof} First observe that to avoid directed cycles, the partial orientation

\vspace{-0.8cm}

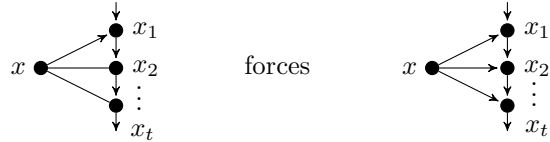
\begin{center}
\begin{tabular}{ccc}

\begin{tikzpicture}[node distance=0.5cm,auto,main node/.style={fill,circle,draw,inner sep=0pt,minimum size=5pt}]

\node[main node] (1) {};
\node [above of=1,yshift=-5mm,xshift=-3mm] {$x$};
\node[main node] (2) [right of=1,xshift=5mm] {};
\node[main node] (3) [above of=2] {};
\node [right of=2,xshift=-1mm] {$x_2$};
\node [right of=3,xshift=-1mm] {$x_1$};
\node [right of=5] {$x_t$};
\node [above right of=5,xshift=1mm,yshift=2mm] {$\vdots$};
\node (4) [above of=3] {};
\node[main node] (5) [below of=2] {};
\node (6) [below of=5] {};

\node (7) [above of=4,yshift=-4mm,xshift=1mm] {};
\node (8) [above right of=7,xshift=6mm]{};

\node (10) [below of=6,yshift=4mm,xshift=1mm] {};
\node (11) [below right of=10,xshift=6mm]{};

\path
(4) [->,>=stealth', shorten >=1pt] edge (3)
(3) edge (2)
(2) edge (5)
(5) edge (6);

\path
(1) [->,>=stealth', shorten >=1pt] edge (3);
\path
(1) edge (2)
(1) edge (5);

\end{tikzpicture}

&

\begin{tikzpicture}[->,>=stealth', shorten >=1pt,node distance=0.5cm,auto,main node/.style={fill,circle,draw,inner sep=0pt,minimum size=5pt}]

\node (1) {};
\node (2) [above of=1,yshift=1cm,xshift=-1cm] {forces};
\end{tikzpicture}

&

\begin{tikzpicture}[->,>=stealth', shorten >=1pt,node distance=0.5cm,auto,main node/.style={fill,circle,draw,inner sep=0pt,minimum size=5pt}]

\node[main node] (1) {};
\node [above of=1,yshift=-5mm,xshift=-3mm] {$x$};
\node[main node] (2) [right of=1,xshift=5mm] {};
\node[main node] (3) [above of=2] {};
\node [right of=2,xshift=-1mm] {$x_2$};
\node [right of=3,xshift=-1mm] {$x_1$};
\node [right of=6,xshift=-1mm,yshift=2mm] {$x_t$};
\node [above right of=5,yshift=-1mm] {$\vdots$};
\node (4) [above of=3] {};
\node[main node] (5) [below of=2] {};
\node (6) [below of=5] {};

\node (7) [above of=4,yshift=-4mm,xshift=1mm] {};
\node (8) [above right of=7,xshift=6mm]{};

\node (10) [below of=6,yshift=4mm,xshift=1mm] {};
\node (11) [below right of=10,xshift=6mm]{};

\path
(4) edge (3)
(3) edge (2)
(2) edge (5)
(5) edge (6);

\path
(1) edge (2)
(1) edge (3)
(1) edge (5);

\end{tikzpicture}
\end{tabular}
\end{center}

\vspace{-0.8cm}

Moreover, the vertices $x_1$, $x_2,\ldots, x_t$ must be consecutive on $\vec{P}$. Indeed,  if $x_i$ and $x_{i+1}$ are not consecutive for some $i$, $1\leq i\leq t-1$ (there is a vertex on the path between $x_i$ and $x_{i+1}$ not connected to $x$) then the vertices on $\vec{P}$ between $x_1$ and $x_{i+1}$, along with $x$,  form a shortcut with the shortcutting edge $x\rightarrow x_{i+1}$.

On the other hand, the partial orientation
\vspace{-0.8cm}

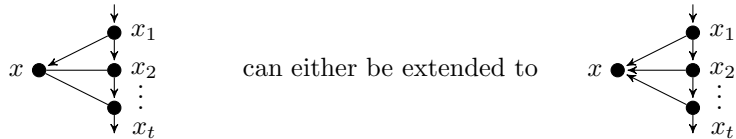
\begin{center}
\begin{tabular}{ccc}

\begin{tikzpicture}[node distance=0.5cm,auto,main node/.style={fill,circle,draw,inner sep=0pt,minimum size=5pt}]

\node[main node] (1) {};
\node [above of=1,yshift=-5mm,xshift=-3mm] {$x$};
\node[main node] (2) [right of=1,xshift=5mm] {};
\node[main node] (3) [above of=2] {};
\node [right of=2,xshift=-1mm] {$x_2$};
\node [right of=3,xshift=-1mm] {$x_1$};
\node [right of=6,xshift=-1mm,yshift=2mm] {$x_t$};
\node [above right of=5,yshift=-1mm] {$\vdots$};
\node (4) [above of=3] {};
\node[main node] (5) [below of=2] {};
\node (6) [below of=5] {};

\node (7) [above of=4,yshift=-4mm,xshift=1mm] {};
\node (8) [above right of=7,xshift=6mm]{};

\node (10) [below of=6,yshift=4mm,xshift=1mm] {};
\node (11) [below right of=10,xshift=6mm]{};

\path
(4) [->,>=stealth', shorten >=1pt] edge (3)
(3) edge (2)
(2) edge (5)
(5) edge (6);

\path
(3) [->,>=stealth', shorten >=1pt] edge (1);
\path
(1) edge (2)
(1) edge (5);

\end{tikzpicture}

&

\begin{tikzpicture}[->,>=stealth', shorten >=1pt,node distance=0.5cm,auto,main node/.style={fill,circle,draw,inner sep=0pt,minimum size=5pt}]

\node (1) {};
\node (2) [above of=1,yshift=1cm,xshift=-1cm] {can either be extended to};
\end{tikzpicture}

&

\begin{tikzpicture}[->,>=stealth', shorten >=1pt,node distance=0.5cm,auto,main node/.style={fill,circle,draw,inner sep=0pt,minimum size=5pt}]

\node[main node] (1) {};
\node [above of=1,yshift=-5mm,xshift=-3mm] {$x$};
\node[main node] (2) [right of=1,xshift=5mm] {};
\node[main node] (3) [above of=2] {};
\node [right of=2,xshift=-1mm] {$x_2$};
\node [right of=3,xshift=-1mm] {$x_1$};
\node [right of=6,xshift=-1mm,yshift=2mm] {$x_t$};
\node [above right of=5,yshift=-1mm] {$\vdots$};
\node (4) [above of=3] {};
\node[main node] (5) [below of=2] {};
\node (6) [below of=5] {};

\node (7) [above of=4,yshift=-4mm,xshift=1mm] {};
\node (8) [above right of=7,xshift=6mm]{};

\node (10) [below of=6,yshift=4mm,xshift=1mm] {};
\node (11) [below right of=10,xshift=6mm]{};

\path
(4) edge (3)
(3) edge (2)
(2) edge (5)
(5) edge (6);

\path
(2) edge (1)
(3) edge (1)
(5) edge (1);

\end{tikzpicture}
\end{tabular}
\end{center}

\vspace{-0.8cm}

\noindent 
with $x_1$, $x_2,\ldots,x_t$ being consecutive to avoid $x_1\rightarrow x$ being a shortcutting edge (by the reasons similar to the previous case), or to avoid directed cycles, all edges of the form $x_i\rightarrow x$ must be above of all edges of the form $x\rightarrow x_i$:
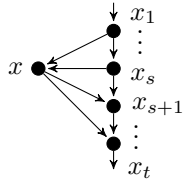
\vspace{-0.8cm}
\begin{center}
\begin{tikzpicture}[->,>=stealth', shorten >=1pt,node distance=0.5cm,auto,main node/.style={fill,circle,draw,inner sep=0pt,minimum size=5pt}]

\node[main node] (1) {};
\node [above of=1,yshift=-5mm,xshift=-3mm] {$x$};
\node[main node] (2) [right of=1,xshift=5mm] {};
\node[main node] (3) [above of=2] {};
\node [right of=2,xshift=-1mm,yshift=-1mm] {$x_s$};
\node [right of=3,xshift=-1mm,yshift=2mm] {$x_1$};
\node [above right of=2,yshift=1mm] {$\vdots$};
\node [right of=6,yshift=5mm,xshift=1mm] {$x_{s+1}$};
\node (4) [above of=3] {};
\node[main node] (5) [below of=2] {};
\node[main node] (6) [below of=5] {};
\node [right of=6,yshift=2mm,xshift=-2mm] {$\vdots$};
\node [below right of=6] {$x_t$};

\node (7) [above of=4,yshift=-4mm,xshift=1mm] {};
\node (8) [above right of=7,xshift=6mm]{};

\node (10) [below of=6,yshift=4mm,xshift=1mm] {};
\node (11) [below right of=10,xshift=6mm]{};
\node (12) [below of=6]{};

\path
(4) edge (3)
(3) edge (2)
(2) edge (5)
(5) edge (6)
(6) edge (12);

\path
(2) edge (1)
(3) edge (1)
(1) edge (5)
(1) edge (6);

\end{tikzpicture}
\end{center}

One can use arguments as above to see that to avoid shortcuts, the vertices $x_1$, $x_2,\ldots,x_s$ corresponding to the edges oriented towards the vertex $x$ must be consecutive on $\vec{P}$. So must be the vertices $x_{s+1}$, $x_{s+2},\ldots,x_t$. On the other hand, there are no restrictions on the vertices $x_s$ and $x_{s+1}$, so there can be some other vertices there on the path $\vec{P}$. 

To complete the theorem, we show that $x_1$ (resp., $x_t$) must be the source (resp., sink) in $\vec{P}$.
Indeed, supposed there exists a vertex $y$ on $\vec{P}$ such that $y\rightarrow x_1$ is an edge. Then the subgraph induced by the vertices $y, x_1, x, x_{s+1}$ is a shortcut with the shortcutting edge $y\rightarrow x_{s+1}$ (because the edge $y\rightarrow x$ is missing); contradiction. Similarly, if there exist a vertex $z$ on $\vec{P}$ such that $x_t\rightarrow z$ is an edge, then the graph induced by the vertices $x_1, x, x_t, z$ is a shortcut with the shortcutting edge $x_1\rightarrow z$ (becayse the edge $x\rightarrow z$ is missing); contradiction.
\end{proof}

\begin{remark} Adding to the path $\vec{P}$ the edge connecting the source to the sink, we obtain a cycle $\vec{C}$ in which one edge is oriented not in the same way as the others. But then all types of vertices in $E_{n-m}$ can be expressed in terms of single consecutive intervals of vertices in $\vec{C}$. Indeed, types A and B stay the same, while type C can now be schematically represented as 
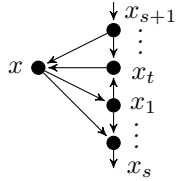
\vspace{-0.8cm}
\begin{center}
\begin{tikzpicture}[->,>=stealth', shorten >=1pt,node distance=0.5cm,auto,main node/.style={fill,circle,draw,inner sep=0pt,minimum size=5pt}]

\node[main node] (1) {};
\node [above of=1,yshift=-5mm,xshift=-3mm] {$x$};
\node[main node] (2) [right of=1,xshift=5mm] {};
\node[main node] (3) [above of=2] {};
\node [right of=2,xshift=-1mm,yshift=-1mm] {$x_t$};
\node [right of=3,yshift=2mm] {$x_{s+1}$};
\node [above right of=2,yshift=1mm] {$\vdots$};
\node [right of=6,yshift=5mm,xshift=-1mm] {$x_1$};
\node (4) [above of=3] {};
\node[main node] (5) [below of=2] {};
\node[main node] (6) [below of=5] {};
\node [right of=6,yshift=2mm,xshift=-2mm] {$\vdots$};
\node [below right of=6] {$x_s$};

\node (7) [above of=4,yshift=-4mm,xshift=1mm] {};
\node (8) [above right of=7,xshift=6mm]{};

\node (10) [below of=6,yshift=4mm,xshift=1mm] {};
\node (11) [below right of=10,xshift=6mm]{};
\node (12) [below of=6]{};

\path
(4) edge (3)
(3) edge (2)
(5) edge (2)
(5) edge (6)
(6) edge (12);

\path
(2) edge (1)
(3) edge (1)
(1) edge (5)
(1) edge (6);

\end{tikzpicture}
\end{center}
\noindent
where the vertices $x_{s+1}$, $x_{s+2},\ldots, x_t,x_1,x_2,\ldots,x_s$ are consecutive.  
\end{remark}

There are additional restrictions on relative positions of the neighbours of vertices of the types A, B and C. These restrictions are given by the following theorem.

\begin{theorem}\label{relative-order} Let $S_n=(E_{n-m},K_m)$ be oriented semi-transitively. For a vertex $x\in E_{n-m}$ of the type C, presented schematically as
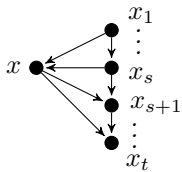
\vspace{-0.8cm}
\begin{center}
\begin{tikzpicture}[->,>=stealth', shorten >=1pt,node distance=0.5cm,auto,main node/.style={fill,circle,draw,inner sep=0pt,minimum size=5pt}]

\node[main node] (1) {};
\node [above of=1,yshift=-5mm,xshift=-3mm] {$x$};
\node[main node] (2) [right of=1,xshift=5mm] {};
\node[main node] (3) [above of=2] {};
\node [right of=2,xshift=-1mm,yshift=-1mm] {$x_s$};
\node [right of=3,xshift=-1mm,yshift=2mm] {$x_1$};
\node [above right of=2,yshift=1mm] {$\vdots$};
\node [right of=6,yshift=5mm,xshift=1mm] {$x_{s+1}$};
\node (4) [above of=3] {};
\node[main node] (5) [below of=2] {};
\node[main node] (6) [below of=5] {};
\node [right of=6,yshift=2mm,xshift=-2mm] {$\vdots$};
\node [below right of=6,yshift=1mm] {$x_t$};

\node (7) [above of=4,yshift=-4mm,xshift=1mm] {};
\node (8) [above right of=7,xshift=6mm]{};

\node (10) [below of=6,yshift=4mm,xshift=1mm] {};
\node (11) [below right of=10,xshift=6mm]{};

\path
(3) edge (2)
(2) edge (5)
(5) edge (6);

\path
(2) edge (1)
(3) edge (1)
(1) edge (5)
(1) edge (6);

\end{tikzpicture}
\end{center}

\noindent 
there is no vertex $y\in E_{n-m}$ of the type A or B, which is connected to both $x_s$ and $x_{s+1}$. Also, there is no vertex $y\in E_{n-m}$ of the type C such that either the source-group, or the sink-group of vertices given by $y$ (see the statement of Theorem~\ref{semi-tran-groups} for the definitions) contains both $x_s$ and $x_{s+1}$.
\end{theorem}

\begin{proof} If $y$ is of the type A, then the subgraph induced by the vertices $y$, $x_s$, $x$ and $x_{s+1}$ is a shortcut with the shortcutting edge being $y\rightarrow x_{s+1}$ (the edge $y\rightarrow x$ is missing). 

Similarly, if $y$ is of the type B, then the subgraph induced by the vertices $y$, $x_s$, $x$ and $x_{s+1}$ is a shortcut with the shortcutting edge being $x_s\rightarrow y$ (the edge $x\rightarrow y$ is missing).  

If $y$ is of the type C and both $x_s$ and $x_{s+1}$ belong to the same group of $y$'s neighbours, then $x_1\rightarrow x_t$ will be a shortcutting edge. Indeed, if both $x_s$ and $x_{s+1}$ belong to 
\begin{itemize} 
\item the source-group then $x_1\rightarrow x_s \rightarrow x \rightarrow x_{s+1} \rightarrow y \rightarrow x_t$ induces a non-transitive subgraph (the edge $y\rightarrow x$ is missing).
\item the sink-group then $x_1\rightarrow y \rightarrow x_s \rightarrow x  \rightarrow x_t$ induces a non-transitive subgraph (the edge $y\rightarrow x$ is missing).
\end{itemize}
\end{proof}

The following theorem is a classification theorem for semi-transitive orientations on split graphs.

\begin{theorem}\label{main-orientation} An orientation of a split graph $S_n=(E_{n-m},K_m)$ is semi-transitive if and only if 
\begin{itemize} 
\item $K_m$ is oriented transitively,
\item each vertex in $E_{n-m}$ is of one of the three types presented in Figure~\ref{3-groups}, and 
\item the restrictions in Theorem~\ref{relative-order} are satisfied. 
\end{itemize}\end{theorem}

\begin{proof} The forward direction follows from Lemma~\ref{lem-tran-orie}, Theorems~\ref{semi-tran-groups} and~\ref{relative-order}.

For the opposite direction, suppose that all restrictions are satisfied, but a shortcut is created with the longest path $\vec{X}$ from the source to the sink. Note that $\vec{X}$ must involve a node in $E_{n-m}$ because $K_m$ is oriented transitively. Also, $\vec{X}$ cannot involve more that one vertex of the type $A$ or $B$ because otherwise we obtain a contradiction with the beginning of $\vec{X}$ not being the beginning of the shortcutting edge (vertices of the type $A$ or $B$ are sinks or sources). 

Next, we note that $\vec{X}$ cannot pass through two vertices of the type $C$ if they satisfy the conditions of Theorem~\ref{relative-order}, which is easy to see from the following two figures representing schematically all possibilities: 

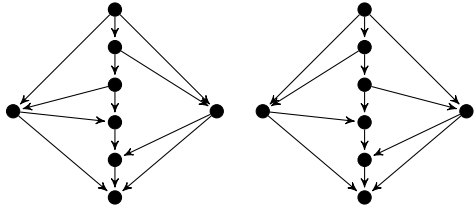
\begin{center}
\begin{tabular}{cc}
\begin{tikzpicture}[->,>=stealth', shorten >=1pt,node distance=0.5cm,auto,main node/.style={fill,circle,draw,inner sep=0pt,minimum size=5pt}]

\node[main node] (1) {};
\node[main node] (2) [below of=1] {};
\node[main node] (3) [below of=2] {};
\node[main node] (4) [below left of=3,xshift=-1cm] {};
\node[main node] (5) [below right of=3,xshift=1cm] {};
\node[main node] (6) [below of=3] {};
\node[main node] (7) [below of=6] {};
\node[main node] (8) [below of=7] {};

\path
(1) edge (2)
(1) edge (4)
(1) edge (5);

\path
(2) edge (3)
(2) edge (5)
(3) edge (6)
(3) edge (4)
(6) edge (7);

\path
(7) edge (8)
(5) edge (7)
(5) edge (8)
(4) edge (6)
(4) edge (8);
\end{tikzpicture}

&

\begin{tikzpicture}[->,>=stealth', shorten >=1pt,node distance=0.5cm,auto,main node/.style={fill,circle,draw,inner sep=0pt,minimum size=5pt}]

\node[main node] (1) {};
\node[main node] (2) [below of=1] {};
\node[main node] (3) [below of=2] {};
\node[main node] (4) [below left of=3,xshift=-1cm] {};
\node[main node] (5) [below right of=3,xshift=1cm] {};
\node[main node] (6) [below of=3] {};
\node[main node] (7) [below of=6] {};
\node[main node] (8) [below of=7] {};

\path
(1) edge (2)
(1) edge (4)
(1) edge (5);

\path
(2) edge (3)
(3) edge (5)
(3) edge (6)
(2) edge (4)
(6) edge (7);

\path
(7) edge (8)
(5) edge (7)
(5) edge (8)
(4) edge (6)
(4) edge (8);
\end{tikzpicture}

\end{tabular}
\end{center}

Finally, we need to consider the situations when $\vec{X}$ passes through 
\begin{itemize}
\item a vertex $y$ of type $A$ and a vertex $x$ of type $C$, and
\item a vertex $y$ of type $B$ and a vertex $x$ of type $C$
\end{itemize}
while respecting the conditions of Theorem~\ref{relative-order}. In either of these cases, both the shortcutting edge and the beginning of $\vec{X}$ must clearly start, or end, at $y$ (depending on $y$'s type). But then, in order for $\vec{X}$ to visit $x$, the vertex $y$ must be connected to both $x_s$ and $x_{s+1}$ in the terminology of Theorem~\ref{relative-order}; contradiction. \end{proof}

As a corollary of Theorem~\ref{main-orientation}, we can provide an alternative proof of the known result that the graph $T_3$ in Figure~\ref{nonRepTri} is non-word-representable. We provide just a sketch of the proof.   
 
\begin{corollary} The graph $T_3$ is non-word-representable. \end{corollary}

\begin{proof} Suppose that a semi-transitive orientation of $T_3$ exists, so that vertices in the independent set are of the type A, or type B, or type C. Then the 4-clique must induce a transitive orientation. Further, out of the vertices in the independent set, at least one, and at most two vertices are of the type $C$. In either of these cases, we obtain a contradiction with the restriction given by Theorem~\ref{relative-order}. \end{proof}

The following corollary of Theorem~\ref{main-orientation} generalizes Theorem~\ref{thm-K-Tri-m-w-r} (which is the case $k=2$ in the corollary). 

\begin{corollary}\label{K-ell-k} Let the split graph $K_{\ell}^k$ be obtained from the complete graph $K_{\ell}$, whose vertices are drawn on a circle, by adding $\ell$ vertices so that 
\begin{itemize}
\item each such vertex is connected to $k$ consecutive (on the circle) vertices in $K_{\ell}$; 
\item neighbourhoods of all these vertices are distinct; and
\item $\ell\geq 2k-1$.
\end{itemize}
Then $K_{\ell}^k$ is word-representable. \end{corollary}

\begin{proof} Orient the clique in $K_{\ell}^k$ transitively with the Hamiltonian path going around the circle, and then assign to the vertices in the independent set types A (or B) and C. Because  $\ell\geq 2k-1$, no vertex of the type A or B will be violating the condition of  Theorem~\ref{relative-order}, and thus by Theorem~\ref{main-orientation}, the obtained orientation is semi-transitive. \end{proof}

Yet another corollary of Theorem~\ref{main-orientation} is a quick proof of non-word-representability of the graph $A_{\ell}$ in Theorem~\ref{Aell-min-non-repres}. Indeed, the fact that the vertex $\ell$ is not connected to any non-clique triangle means that it is either a sink/source, or there exists (exactly one) vertex of the type C in the independent set. In the former case, no type C vertices can exist, and one can get a contradiction with a vertex in the independent set be connected to non-consecutive vertices in $\vec{P}$. Same contradiction is obtained in the later case. 

We complete this section with one more property of semi-transitive orientations on split graphs. 

\begin{theorem}\label{intercahnging} 
Let $S_n=(E_{n-m},K_m)$ be semi-transitively oriented. Then any vertex in $E_{n-m}$ of the type A can be replaced by a vertex of the type B, and vice versa, keeping orientation be semi-transitive. 
\end{theorem}

\begin{proof} 
Suppose that a vertex $x$ of the type A 
\vspace{-0.8cm}
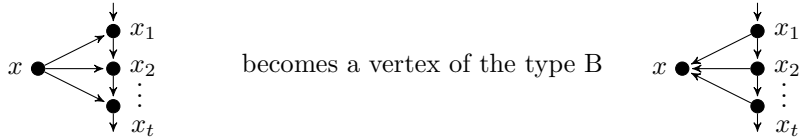
\begin{center}
\begin{tabular}{ccc}

\begin{tikzpicture}[->,>=stealth', shorten >=1pt,node distance=0.5cm,auto,main node/.style={fill,circle,draw,inner sep=0pt,minimum size=5pt}]

\node[main node] (1) {};
\node [above of=1,yshift=-5mm,xshift=-3mm] {$x$};
\node[main node] (2) [right of=1,xshift=5mm] {};
\node[main node] (3) [above of=2] {};
\node [right of=2,xshift=-1mm] {$x_2$};
\node [right of=3,xshift=-1mm] {$x_1$};
\node (4) [above of=3] {};
\node[main node] (5) [below of=2] {};
\node (6) [below of=5] {};
\node [right of=6,xshift=-1mm,yshift=2mm] {$x_t$};
\node [above right of=5,yshift=-1mm] {$\vdots$};
\node (7) [above of=4,yshift=-4mm,xshift=1mm] {};
\node (8) [above right of=7,xshift=6mm]{};

\node (10) [below of=6,yshift=4mm,xshift=1mm] {};
\node (11) [below right of=10,xshift=6mm]{};

\path
(4) edge (3)
(3) edge (2)
(2) edge (5)
(5) edge (6);

\path
(1) edge (2)
(1) edge (3)
(1) edge (5);

\end{tikzpicture}

&

\begin{tikzpicture}[->,>=stealth', shorten >=1pt,node distance=0.5cm,auto,main node/.style={fill,circle,draw,inner sep=0pt,minimum size=5pt}]

\node (1) {};
\node (2) [above of=1,yshift=1cm,xshift=-1cm] {becomes a vertex of the type B};
\end{tikzpicture}

&

\begin{tikzpicture}[->,>=stealth', shorten >=1pt,node distance=0.5cm,auto,main node/.style={fill,circle,draw,inner sep=0pt,minimum size=5pt}]

\node[main node] (1) {};
\node [above of=1,yshift=-5mm,xshift=-3mm] {$x$};
\node[main node] (2) [right of=1,xshift=5mm] {};
\node[main node] (3) [above of=2] {};
\node [right of=2,xshift=-1mm] {$x_2$};
\node [right of=3,xshift=-1mm] {$x_1$};
\node [right of=6,xshift=-1mm,yshift=2mm] {$x_t$};
\node [above right of=5,yshift=-1mm] {$\vdots$};
\node (4) [above of=3] {};
\node[main node] (5) [below of=2] {};
\node (6) [below of=5] {};

\node (7) [above of=4,yshift=-4mm,xshift=1mm] {};
\node (8) [above right of=7,xshift=6mm]{};

\node (10) [below of=6,yshift=4mm,xshift=1mm] {};
\node (11) [below right of=10,xshift=6mm]{};

\path
(4) edge (3)
(3) edge (2)
(2) edge (5)
(5) edge (6);

\path
(2) edge (1)
(3) edge (1)
(5) edge (1);

\end{tikzpicture}
\end{tabular}
\end{center}

\vspace{-0.8cm}
\noindent
while no other orientation is changed in the semi-transitively oriented $S_n$. Clearly, if the change has resulted in a non-semi-transitive orientation then the vertex $x$ must be involved in a shortcut (it cannot be involved in a directed cycle) with $x_i\rightarrow x$ being a shortcutting edge for some $i$. This contradicts to the vertices $x_1,\ldots,x_t$ being consecutive on $\vec{P}$ and inducing a transitive orientation together with $x$.

Essentially identical arguments, with a shortcutting edge being $x\rightarrow x_i$ this time, show that switching from type B to type A for a vertex $x\in E_{n-m}$ does not result in a non-semi-transitive orientation.
\end{proof}

\section{Concluding remarks}\label{last-sec}

In this paper, we characterized in terms of forbidden subgraphs word-representable split graphs $S_n=(E_{n-m},K_m)$ in which vertices in $E_{n-m}$ are of degree at most 2 (see Theorem~\ref{main-1}), or the size of $K_m$ is 4 (see Theorem~\ref{main-2}). Moreover, in Theorem~\ref{main-orientation} we give necessary and sufficient conditions for an orientation of a split graph to be semi-transitive.

There are two natural directions in which our results could be extended. One can consider vertices of degree at most 3 in $E_{n-m}$ (thus extending the results in Theorem~\ref{main-1}), or letting the clique be $K_5$ (thus extending the results in Theorem~\ref{main-2}). Either of these directions is challenging due to a large number of cases to consider. It is conceivable that our classification result, Theorem~\ref{main-orientation}, on semi-transitive orientations of split graphs will eventually be the key for a complete classification of word-representable split graphs. 

\section{Acknowledgments}
The first author is grateful to the Fudan University and Shanghai Jiao Tong University for their hospitality during the author's visit of the places in May 2017. The second author acknowledges support of National Natural Science Foundation of China (No. 11671258) and Postdoctoral Science Foundation of China  (No.\ 2016M601576).

\end{document}